\documentclass[a4paper,11pt, reqno]{amsart}
\usepackage{etex}
\reserveinserts{28}
\usepackage[T1]{fontenc}
\usepackage{silence,lmodern}
\usepackage[utf8]{inputenc}
\usepackage[english]{babel}
\usepackage[stretch=10]{microtype}
\usepackage{mathtools}
\usepackage{amsthm}
\usepackage{amssymb}
\usepackage{url}
\usepackage{enumerate}
\usepackage{wrapfig}
\usepackage{bbm}
\usepackage{color,tikz,pgfplots}
\usepackage{pst-func}
\usepackage{graphicx}
\usepackage{caption}
\usepackage{subcaption}
\usepackage{float}
\usepackage{csquotes}
\usepackage[left=3cm, top=3cm,bottom=3cm,right=3cm]{geometry}

\newtheorem{theorem}{Theorem}[section]

\newtheorem{lemma}[theorem]{Lemma}
\newtheorem{proposition}[theorem]{Proposition}
\newtheorem{corollary}[theorem]{Corollary}
\newtheorem{conjecture}[theorem]{Conjecture}
\newtheorem*{theorem*}{Theorem}

\theoremstyle{definition}
\newtheorem{definition}[theorem]{Definition}
\newtheorem{notation}{Notation}
\newtheorem{remark}[theorem]{Remark}
\newtheorem{exmp}{Example}

\newtheorem*{acknowledgements}{Acknowledgements}

\numberwithin{equation}{section}

\newcommand{\cA}{\mathcal{A}}
\newcommand{\cC}{\mathcal{C}}

\DeclareMathOperator{\cL}{\mathbin{Law}}
\DeclareMathOperator{\cME}{\mathtt{ExInHy}}
\DeclareMathOperator{\cMEH}{\mathtt{ExHier}}

\newcommand{\cK}{\mathcal{K}}
\newcommand{\dha}{d_{\mathbin{haus}}}

\newcommand{\cF}{\mathcal{F}}

\newcommand{\cH}{\mathcal{H}}

\newcommand{\hI}{\mathtt{I}}
\newcommand{\hT}{\mathtt{T}}
\newcommand{\ias}{\overset{a.s.}{\Rightarrow}}
\newcommand{\eas}{\overset{a.s.}{\Leftrightarrow}}

\newcommand{\hH}{\mathtt{H}}
\newcommand{\hP}{\mathtt{P}}
\newcommand{\cM}{\mathcal{M}}

\newcommand{\cB}{\mathcal{B}}
\DeclareMathOperator{\cMI}{\mathtt{ErInPr}}
\DeclareMathOperator{\cMB}{\mathtt{ErBTPr}}
\DeclareMathOperator{\cMS}{\mathtt{ErSTPr}}
\DeclareMathOperator{\cMC}{\mathtt{ExComp}}

\newcommand{\bP}{\mathbb{P}}
\DeclareMathOperator{\bIH}{\mathbin{InHy}}
\DeclareMathOperator{\bH}{\mathbin{Hier}}
\DeclareMathOperator{\bIP}{\mathbin{InPar}}
\DeclareMathOperator{\bIS}{\mathbin{InSy}}
\DeclareMathOperator{\bST}{\mathbin{STree}}
\DeclareMathOperator{\bSTi}{\mathbin{STree}(\infty)}

\DeclareMathOperator{\bBT}{\mathbin{BinTree}}
\DeclareMathOperator{\bBTi}{\mathbin{BinTree}(\infty)}
\DeclareMathOperator{\bISi}{\mathbin{InSy}(\infty)}
\DeclareMathOperator{\bIHH}{\mathbin{IHier}([0,1))}
\newcommand{\bR}{\mathbb{R}}
\newcommand{\cP}{\mathcal{P}}
\newcommand{\cS}{\mathcal{S}}
\newcommand{\cD}{\mathcal{D}}
\newcommand{\bS}{\mathbb{S}}
\newcommand{\vj}{\vec{j}}

\newcommand{\bL}{\mathbb{L}}

\newcommand{\bN}{\mathbb{N}}

\DeclareMathOperator{\PH}{\bIH(\bN)}

\pgfplotsset{compat=1.6}

\DeclareMathOperator{\unif}{\mathbin{unif}}

\DeclareMathOperator{\ex}{\mathbin{erg}}

\DeclareMathOperator{\Bin}{\mathbin{Bin}}

\DeclareMathOperator{\law}{\mathbin{Law}}
\DeclareMathOperator{\inte}{\mathbin{int}}
\DeclareMathOperator{\iso}{\mathbin{iso}}

\DeclareMathOperator{\graph}{\mathbin{graph}}

\newcommand{\diago}[1]{
	\begin{tikzpicture}[#1]
	\draw[color=black, thick] (0,0)--(0.3,0.3);
	\end{tikzpicture}
}

\newcommand{\recto}[1]{
	\begin{tikzpicture}[#1]
	\draw[fill=black!20]  (0,0) -- (0.3,0.3) -- (0,0.3) -- cycle;
	\draw[color=black]  (0,0) -- (0.3,0.3) -- (0,0.3) -- cycle;
	\end{tikzpicture}
}

\author[J. Gerstenberg]{Julian Gerstenberg}
\address{Julian Gerstenberg: Institut f\"ur Mathematische Stochastik, Leibniz Universit\"at Hannover, Welfengarten 1, 30167 Hannover, Germany}
\email{jgerst@stochastik.uni-hannover.de}


\begin{document}
	
	\title[exchangeable interval hypergraphs]{Exchangeable interval hypergraphs and limits of ordered discrete structures}
	
	\begin{abstract}
		A hypergraph $(V,E)$ is called an interval hypergraph if there exists a linear order $l$ on $V$ such that every edge $e\in E$ is an interval w.r.t. $l$; we also assume that $\{j\}\in E$ for every $j\in V$. Our main result is a de~Finetti-type representation of random exchangeable interval hypergraphs on~$\bN$ (EIHs): the law of every EIH can be obtained by sampling from some random compact subset $K$ of the triangle $\{(x,y):0\leq x\leq y\leq 1\}$ at iid uniform positions $U_1,U_2,\dots$, in the sense that, restricted to the node set $[n]:=\{1,\dots,n\}$ every non-singleton edge is of the form $e=\{i\in[n]:x<U_i<y\}$ for some $(x,y)\in K$. We obtain this result via the study of a related class of stochastic objects: erased-interval processes (EIPs). These are certain transient Markov chains $(I_n,\eta_n)_{n\in\bN}$ such that $I_n$ is an interval hypergraph on $V=[n]$ w.r.t. the usual linear order (called interval system). We present an almost sure representation result for EIPs. Attached to each transient Markov chain is the notion of Martin boundary. The points in the boundary attached to EIPs can be seen as limits of growing interval systems. We obtain a one-to-one correspondence between these limits and compact subsets $K$ of the triangle with $(x,x)\in K$ for all $x\in[0,1]$.
		
		Interval hypergraphs are a generalizations of hierarchies and as a consequence we obtain a representation result for exchangeable hierarchies, which is close to a result of Forman, Haulk and Pitman in \cite{fohapi}. Several ordered discrete structures can be seen as interval systems with additional properties, i.e. Schr{\"o}der trees (rooted, ordered, no node has outdegree one) or even more special: binary trees. We describe limits of Schr{\"o}der trees as certain tree-like compact sets. These can be seen as an ordered counterpart to real trees, which are widely used to describe limits of discrete unordered trees. Considering binary trees we thus obtain a homeomorphic description of the Martin boundary of R\'emy's tree growth chain, which has been analyzed by Evans, Gr{\"u}bel and Wakolbinger in \cite{egw2}. 
	\end{abstract}
	\maketitle

	\section{introduction}\label{sec:INTRO}
	
	The classical de Finetti representation theorem can be stated as follows: \emph{the law of every exchangeable $\{0,1\}$-valued stochastic processes can be expressed as a mixture of laws of iid processes}. More precisely, for any law $P$ of an exchangeable $\{0,1\}$-valued processes there exists a unique Borel probability measure $\mu$ on $[0,1]$ such that $P=\int_{[0,1]}\Bin(1,p)^{\otimes\bN}d\mu(p)$\footnote{\label{foot:2}$\Bin(n,p)$ denotes the binomial distribution and $Q^{\otimes\bN}=Q\otimes Q\otimes\cdots$ is the law of an iid sequence with marginal distribution $Q$.}. De Finetti's theorem has been generalized in many different directions. With the present paper we contribute to the growing list of de Finetti-type representation theorems by studying \emph{exchangeable interval hypergraphs on $\bN$}. Up to now the list of (combinatorial) structures whose attached exchangeability structure have been analyzed includes: sequences, partitions, graphs, general arrays (see \cite{kallenberg}) and more recently, hierarchies by Forman, Haulk and Pitman \cite{fohapi}. Our work provides a generalization of some of their results, since every exchangeable hierarchy on $\bN$ is an exchangeable interval hypergraph on $\bN$ as well. A more complete list concerning exchangeability in combinatorial objects can be found in \cite[Section 1.2.]{fohapi}.
	
	For any set $V$ let $\cP(V)$ be the power set of $V$. A \emph{hypergraph} is a tuple $\hH=(V,E)$ where $V$ is the set of nodes and $E\subseteq \cP(V)$ is the set of edges. We only consider hypergraphs that contain all singletons, i.e. with $\{j\}\in E$ for every $j\in V$. In such hypergraphs, it is not necessary to specify the set of nodes and therefore we can identify the hypergraph $\hH$ with its set of edges $E$. For $n\in\bN=\{1,2,\dots\}$ let $[n]:=\{1,\dots,n\}$. We now introduce the basic combinatorial structure considered in this paper:
	
	\begin{definition}\label{def:IntervalHypergraph}
		Let $n\in\bN$. A set $\hH\subseteq \cP([n])$ is called an \emph{interval hypergraph on $[n]$}, if 
		\begin{enumerate}
			\item[(i)] $\emptyset\in\hH~\text{and}~\{j\}\in\hH~\text{for all}~j\in[n]$,
			\item[(ii)] there exists a linear order $l$ on $[n]$ such that every $e\in \hH$ is an interval\footnote{\label{foot:1}If a linear order $l$ on a set $V$ is given, we write $x~l~y$ if $x$ is w.r.t. $l$ smaller then $y$. An interval with respect to $l$ is a subset $e\subseteq V$ such that for every $x,y\in e$ and $z\in V$ the implication $xlzly\Rightarrow z\in e$ holds.} with respect to $l$.
		\end{enumerate}
		Let $\bIH(n)$ be the set of all interval hypergraphs on $[n]$. If every $e\in\hH$ is an interval w.r.t. $l$, we will say that \emph{$\hH$ is an interval hypergraph w.r.t. $l$}. For any fixed linear order $l$ on $[n]$ let $\bIH(n,l)$ be the set of all interval hypergraphs on $[n]$ w.r.t. $l$. See \cite{Moore} for a slightly different definition and additional material concerning interval hypergraphs and Fig. \ref{fig:inhyexmp} for a visualization.
	\end{definition}

	\begin{wrapfigure}{r}{0.4\textwidth}
		\centering
		\includegraphics[width=0.4\textwidth]{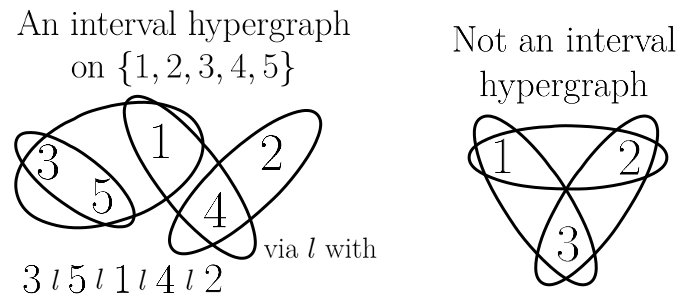}
		\caption{\label{fig:inhyexmp}}
	\end{wrapfigure}

	Next we define interval hypergraphs on $\bN$. Of course, one could replace the set of nodes $[n]$ in Definition \ref{def:IntervalHypergraph} by an arbitrary set $V$ and use the same requirements (i),(ii) to introduce interval hypergraphs on $V$. Instead we define an interval hypergraph on $\bN$ as a \emph{projective sequence of finite interval hypergraphs}, more concretely: Given some $k\leq n$ and some $\hH\in\bIH(n)$ we introduce the \emph{restricted} interval hypergraph
	\begin{equation}\label{eq:restriction}
	\hH_{|k}:=[k]\cap \hH=\left\{\{1,\dots,k\}\cap e:e\in\hH\right\}.
	\end{equation} This really yields an interval hypergraph: if $\hH$ is an interval hypergraph w.r.t. $l$, then $\hH_{|k}$ is an interval hypergraph w.r.t. $l_{|k}$, where $l_{|k}$ is the restriction of the linear order $l$ to the set $[k]$, see Fig. \ref{fig:restrelexmp}.
	\begin{definition}
		A sequence $\hH=(\hH_n)_{n\in\bN}$ such that for every $n$
		\begin{itemize}
			\item[-] $\hH_n\in\bIH(n)$
			\item[-] $\hH_n=(\hH_{n+1})_{|n}$
		\end{itemize}
		is called an \emph{interval hypergraph on $\bN$}. Let $\bIH(\bN)$ be the set of all interval hypergraphs on $\bN$. 
	\end{definition}
	
	\begin{remark}\label{rem:card}
		Consider subsets $\hH',\hH\subseteq \cP(\bN)$ that both satisfy (i) and (ii) as in Definition~\ref{def:IntervalHypergraph} but with $\bN$ instead of $[n]$. It is possible that $\hH\neq \hH'$ but $\hH_{|n}=\hH'_{|n}$ for all $n$, where $\hH_{|n}=\{[n]\cap e:e\in \hH\}$. The set of all $\hH\subseteq\cP(\bN)$ satisfying (i) and (ii) has a cardinality higher than that of the continuum. Thus it can not be equipped with a $\sigma$-field that turns it into a Borel space, which would be a desirable technical feature when dealing with exchangeable random objects. In \cite{fohapi} hierarchies on $\bN$ are defined as projective sequences of finite hierarchies as well. 
	\end{remark}
	
	 A \emph{random interval hypergraph on $\bN$} is a stochastic process $H=(H_n)_{n\in\bN}$ such that $H_n\in\bIH(n)$ and $H_n=(H_{n+1})_{|n}$ almost surely for all $n$. Let $\cL(H)$ be the law of a random interval hypergraph $H$ on $\bN$. To introduce exchangeability we first explain how to relabel interval hypergraphs: Let $\bS_n$ be the group of bijections $\pi:[n]\rightarrow[n]$. The \emph{one-line notation} of $\pi\in\bS_n$ is the vector $(\pi(1),\dots,\pi(n))$ and $\pi^{-1}$ is the \emph{inverse} of $\pi$. Given $\pi\in\bS_n$ and $\hH\in\bIH(n)$ we introduce the \emph{relabeled} interval hypergraph
	 \begin{equation}\label{eq:relabeling}
	 \pi(\hH):=\{\pi(e):e\in\hH\}.
	 \end{equation}
	 This really yields an interval hypergraph: if $\hH$ is an interval hypergraph w.r.t. $l$, then $\pi(\hH)$ is an interval hypergraph w.r.t. $l^{\pi}$, where $il^{\pi}j:\Leftrightarrow \pi^{-1}(i)l\pi^{-1}(j)$, see Fig. \ref{fig:restrelexmp}. Now we can introduce one of our main objects of interest:
	 \begin{definition}\label{def:ExIH}
	 	An \emph{exchangeable interval hypergraph on $\bN$} (EIH) is a random interval hypergraph $H=(H_n)_{n\in\bN}$ on $\bN$ such that all $H_n$ have \emph{exchangeable laws}, i.e. for every $n$ it holds that $\pi(H_n)\overset{\cD}{=}H_n$ for every $\pi\in\bS_n$. Let  
	 	$$\cME=\{\cL(H):\text{$H$ is an exchangeable interval hypergraph on $\bN$}\}$$
	 	be the collection of all possible laws of EIHs.
	 \end{definition}
 
 	We can now state our de Finetti-type representation result for EIHs, ignoring some measurability issues for the moment:
 	\begin{theorem}\label{thm:start}
 		For every exchangeable interval hypergraph $(H_n)_{n\in\bN}$ there exists a random compact set $K\subseteq\{(x,y)\in\bR^2:0\leq x\leq y\leq 1\}$ with $(x,x)\in K$ for every $x\in[0,1]$ such that, given an iid sequence $U_1,U_2,\dots$ of uniform RVs independent of $K$, it holds that $(H_n)_{n\in\bN}$ has the same distribution as $(H'_n)_{n\in\bN}$, where 
 		\begin{equation*}
 		H'_n:=\Big\{\big\{i\in[n]:x<U_i<y\big\}:(x,y)\in K\Big\}\cup\Big\{\{j\}:j\in[n]\Big\}\cup\big\{\emptyset\big\}.
 		\end{equation*}
 	\end{theorem}
 
 	Our aim is not only to describe the laws of EIHs as in Theorem~\ref{thm:start}, but also to describe some topological aspects of the space $\cME$. In particular, $\cME$ is naturally equipped with the structure of a \emph{metrizable Choquet simplex}, i.e. it is a compact convex set in which every point can be expressed as a mixture of extreme points in a unique way. We will not only describe its convexity structure, but also describe its topology. In particular, we will show that $\cME$ is a Bauer simplex, i.e. the extreme points (which are precisely the ergodic exchangeable laws) form a closed set. We will give further details concerning the simplex point of view at the end of this introduction. 
 	
 	\begin{wrapfigure}{r}{0.5\textwidth}
 		\centering
 		\includegraphics[width=0.5\textwidth]{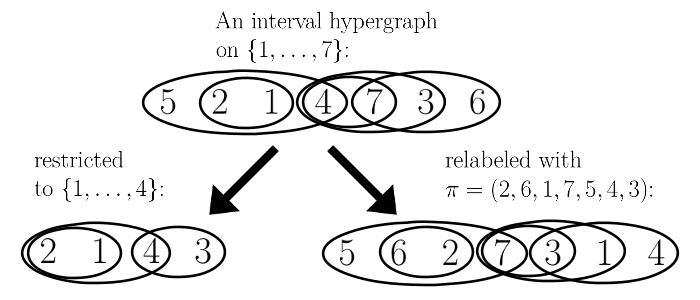}
 		\caption{\label{fig:restrelexmp}}
 	\end{wrapfigure}
 	
	As already mentioned, our main result for EIHs can be seen as an extension of the results concerning exchangeable hierarchies on $\bN$ that are presented in \cite{fohapi}, since every exchangeable hierarchy on $\bN$ is an EIH as well. We explain this connection in detail in Section \ref{sec:connections}. We think that our results for EIHs are interesting, not only because they are a generalization (and a topological refinement) of some results given in \cite{fohapi}, but also because our method of proof is different, as we now explain:
	
	We introduce a different class of stochastic objects, \emph{erased-interval processes} (EIPs), and we present a de Finetti-type representation for these: \emph{not only in law, but almost surely}. We derive our representation result concerning interval hypergraphs from the result for erased-interval processes. Such EIPs are certain transient Markov chains $(I_n,\eta_n)_{n\in\bN}$ where $I_n$ takes values in $\bIH(n,<)$ and $\eta_n$ takes values in $[n+1]$, satisfying a certain dependency structure, which we will define below. We introduce the \emph{Martin boundary attached to erased-interval processes} in Section $5$. The close connection to the work of Evans, Gr{\"u}bel and Wakolbinger stems from the fact that EIPs yield a generalization of \emph{infinite labeled R\'emy bridges} introduced in \cite{egw2}. These authors examined the \emph{Martin boundary of R\'emy's tree growth chain}, which is a transient Markov chain $(T_n)_{n\in\bN}$ such that $T_n$ is a uniform random binary tree with exactly $n$ leaves. To achieve a description of the Martin boundary associated to that particular Markov chain, the authors reduced the problem of describing the boundary to the task of examining the collection of laws of infinite labeled R\'emy bridges. The set of binary trees with $n$ leafs can be considered a subset $\bBT(n)\subsetneq\bIH(n,<)$ and infinite labeled R\'emy Bridges are basically those erased-interval processes $(I_n,\eta_n)_{n\in\bN}$ with the additional property that $I_n\in\bBT(n)$ for every $n$. Our results concerning erased-interval processes can be used to obtain homeomorphic descriptions of certain Martin boundaries and in particular, we present a homeomorphic description of the Martin boundary of R\'emy's tree growth chain. 
	In Section~\ref{sec:connections} we connect the contents of \cite{fohapi} and \cite{egw2} thereby explaining some of their similarities: in both of these works exchangeable random objects were represented by some sort of \emph{sampling from real trees}. We will will follow a different route: Our more general underlying discrete objects (interval hypergraphs and interval systems defined below) no longer bear tree-like structures and so, in general, trees do not appear in the associated representation results. Instead of a tree-like structure to work with, our discrete objects can be embedded into the upper triangle 
	$$\recto{}:=\Big\{(x,y)\in\bR^2:0\leq x\leq y\leq 1\Big\}$$
	in an almost canonical way. Ordered trees later appear as subclasses of interval systems and we obtain results for trees as corollaries to our main theorem concerning EIPs.
	
	Erased-type processes occur in the literature concerning \emph{poly-adic filtrations}, see \cite{laurent2016filtrations} and \cite{gerstenberg}. Certain results in these papers later yield an explanation for why we consider only hypergraphs in which all singleton sets $\{j\}$ are part of the edge set. Our almost-sure representation result for EIPs, as a by-product, also clarifies the isomorphism structure of the poly-adic backward filtrations generated by those processes: such a backward filtration $\cF=(\cF_n)_{n\in\bN}$ is of product-type iff it is Kolmogorovian, i.e. if the terminal $\sigma$-field $\cF_{\infty}=\bigcap_{n\in\bN}\cF_n$ is a.s. trivial.
	
	Recall that we use $<$ for the usual linear order of real numbers. For $n\in\bN$ we call an $\hI\in\bIH(n,<)$ an \emph{interval system} on $[n]$. In particular, every non-empty edge in an interval system is of the form $e=[a,b]:=\{i\in[n]:a\leq i\leq b\}$ for some $1\leq a\leq b\leq n$. Denote the space of all interval systems on $[n]$ by $\bIS(n):=\bIH(n,<)$, so $\bIS(n)\subseteq \bIH(n)$. See Fig. \ref{fig:eipexmp} for some visualizations of interval systems. The relation between interval systems and interval hypergraphs is not just that the former is more special than the latter: for \emph{every} interval hypergraph $\hH$ on $[n]$ there exists some permutation $\pi\in\bS_n$ such that the relabeled interval hypergraph $\pi(\hH)$ is an interval system. This property later allows to transfer our result concerning EIPs to exchangeable interval hypergraphs on $\bN$.
	
	The definition of exchangeable interval hypergraphs on $\bN$ was built upon restricting and relabeling. For erased-interval processes we introduce a different operation: \emph{removing elements from intervals and then relabeling them in a strictly monotone fashion}. Let $n\in\bN$, $[a,b]\subseteq[n+1]$ and $k\in[n+1]$. We define an interval $[a,b]-\{k\}\subseteq[n]$ by
	\begin{equation}\label{eq:intervalrestr}
	[a,b]-\{k\}:=\begin{cases}
	[a-1,b-1],&\text{if}~k<a\leq b\\
	[a,b-1],&\text{if}~a\leq k\leq b,\\
	[a,b],&\text{if}~a\leq b<k.
	\end{cases}
	\end{equation}
	This operation can be lifted to interval systems: If $\hI\in\bIS(n+1)$ is an interval system on $[n+1]$ and $k\in [n+1]$, then $\phi^{n+1}_n(\hI,k)$ is defined by removing $k$ from every $[a,b]\in\hI$ in the sense of (\ref{eq:intervalrestr}). So $\phi^{n+1}_n$ is formally defined as
	\begin{align*}
	\phi^{n+1}_n:\bIS(n+1)\times[n+1]&\longrightarrow\bIS(n),\\
	&\phi^{n+1}_n(\hI,k):=\Big\{[a,b]-\{k\}:[a,b]\in\hI\Big\}\cup\{\emptyset\}.
	\end{align*} 
	
	Now we are ready to introduce erased-interval processes:
	
	\begin{definition}\label{def:EIP}
		An \emph{erased-interval process} is a stochastic process $(I,\eta)=(I_n,\eta_n)_{n\in\bN}$ such that for every $n\in\bN$:
		\begin{enumerate}
			\item[(i)] $(I_n,\eta_n)$ takes values in $\bIS(n)\times[n+1]$ almost surely.
			\item[(ii)] The random variable $\eta_n$ is uniformly distributed on $[n+1]$ and independent of the $\sigma$-field $\cF_{n+1}$, where 
			$$\cF_n:=\sigma\left(I_m,\eta_m:m\geq n\right)$$
			is the $\sigma$-field generated by the future of $(I,\eta)$ after time $n$.
			\item[(iii)] $I_n=\phi^{n+1}_n(I_{n+1},\eta_n)$ almost surely. 
		\end{enumerate}
		Let $\cL((I,\eta))$ be the law of the erased-interval process $(I,\eta)$ and let
		\begin{equation*}
		\cMI:=\left\{\cL((I,\eta)):(I,\eta)~\text{is an erased-interval process}\right\}
		\end{equation*}
		be the space of all possible laws of erased-interval processes. See Fig. \ref{fig:eipexmp} for a visualization of property (iii).
	\end{definition}

	\begin{wrapfigure}{r}{0.46\textwidth}
		\centering
		\includegraphics[width=0.46\textwidth]{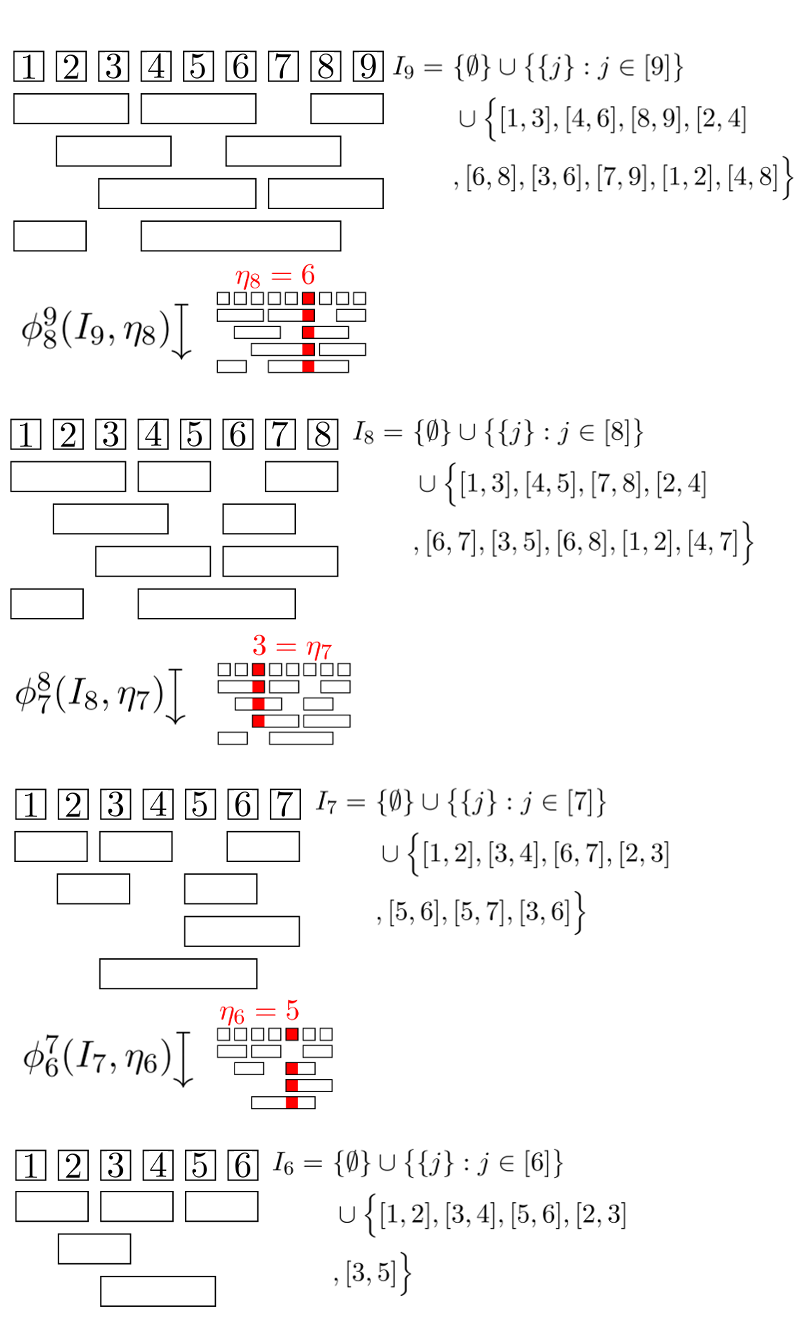}
		\caption{\label{fig:eipexmp}}
	\end{wrapfigure}

	In Section \ref{sec:mainresults} we present a de Finetti-type representation result for EIPs in a very 'strong' form: we not only describe every possible law of EIPs in a certain way, but we express \emph{any given} EIP in a certain \emph{almost sure} way. We also describe the whole space $\cMI$, which like $\cME$ is a metrizable Choquet simplex (see the end of this introduction), up to affine homeomorphism. We will deduce our representation result for exchangeable interval hypergraphs by providing an explicit affine surjective map $\cMI\rightarrow\cME$. 
	
	We think that topological aspects concerning exchangeability and related fields are interesting, since they generalize the following considerations: The exchangeability structure of $\{0,1\}$-valued processes is answered by the famous saying that every such process is a mixture of iid processes (de Finetti's theorem). The latter are parametrized by the set $[0,1]$, where the parameter $p\in[0,1]$ is the success probability. The unit interval can be considered, as usual, as a topological space and this topology is the right one: the map $p\mapsto \Bin(1,p)^{\otimes\bN}$ from the unit interval to the set of all laws of iid $\{0,1\}$-valued processes is not only bijective, but a \emph{homeomorphism}, if the latter space is equipped with its natural topology of weak convergence. This identifies the metrizable Choquet simplex of all laws of exchangeable $\{0,1\}$-valued processes as a Bauer simplex that is affinely homeomorphic to the simplex of Borel probability measures on the unit interval, $\cM_1([0,1])$, equipped with the topology of weak convergence. In the case of $\{0,1\}$-valued processes these topological considerations may not come as a surprise. But in more complex situations, the space $[0,1]$ used to describe the ergodic laws before is replaced by more complex spaces and accurate topological descriptions become a more challenging task.
	
	In Section \ref{sec:proof} we will prove our main theorems. Our proofs differ strongly from the proofs given in \cite{egw2} and \cite{fohapi}, since they both used the tree-like structures present in the discrete objects they considered. Because of that, they were able to introduce some certain indicator sequences or arrays attached to their exchangeable random structures that inherited exchangeability. Then they both used, in some sense, the classical de Finetti theorem to obtain almost sure convergences. Our proofs do not need such results, but only algorithmic features of certain constructions, the Glivenko-Cantelli theorem and a statement concerning the intersection behavior of random rectangles.

	In Section \ref{sec:CUP} we provide a tool which is used in our proof concerning the almost sure representation of EIPs at two crucial points. This tool can loosely be stated as \emph{'some random rectangles intersect a given compact set in its interior or not at all, almost surely'}. This statement is derived from a topological fact concerning the \emph{Sorgenfrey plane}\footnote{\label{foot:3}i.e. the topological space based on $\bR^2$ equipped with the topology generated by rectangles $[a,b)\times[c,d)$.}: Given any subset $A$ of the Sorgenfrey plane the set of isolated points of $A$ is contained in the union of at most countable many strictly decreasing functions. 
	
	In Section \ref{sec:connections} we discuss the relations to the already mentioned papers \cite{fohapi} and \cite{egw2} in detail and point out how to use our results in the situations considered there. We explain the similarities between these works from a structural point of view. We identify Martin boundaries corresponding to interval systems, Schr{\"o}der trees and binary trees. These Martin boundaries can be interpreted as \emph{limits of growing structures}: For $1\leq k\leq n$ and $\hI_k\in\bIS(k),\hI_n\in\bIS(n)$ we will define the \emph{density} of the 'small' object $\hI_k$ in the 'large' object $\hI_n$ by counting all the ways one can \emph{embed} $\hI_k$ in a strictly monotonic way into $\hI_n$. This count then gets divided by the total number of possible embeddings, which is given by $\binom{n}{k}$, so the density of $\hI_k$ in $\hI_n$ is thus some number $\gamma(\hI_k,\hI_n)\in[0,1]$. A (non-random) sequence $(\hI_n)_{n\in\bN}$ with $\hI_n\in\bIS(m_n)$ and $m_n\rightarrow\infty$ is called \emph{$\gamma$-convergent} iff $\gamma(\hI,\hI_n)$ converges for every $\hI\in\bIS(k), k\in\bN$. The function $\hI\mapsto \lim_n\gamma(\hI,\hI_n)$ is then considered to be \emph{the limit} of the $\gamma$-convergent sequence $(\hI_n)_{n\in\bN}$. We will describe the set of all possible limits in a homeomorphic way. Limits of discrete structures in this spirit are closely related to exchangeability in combinatorial objects and very often, these concepts are equivalent, see \cite{dija, austin}. The connection of Martin boundary theory and limits of discrete structures in the case of graphs has been pointed out in \cite{gruebel}. The classical de Finetti theorem has been connected to Martin boundary theory in \cite{GGH}. Further closely related concepts are discussed in \cite{hkmrs} (pattern densities in permutations), \cite{gerstenberg} (subsequence densities in words) and \cite{choi2017doob} (subsequence densities in words in which all letters of the alphabet occur equally often). In Section $5$ we will also briefly consider \emph{compositions}, an additional class of ordered discrete structures embedded in interval systems, and explain how to use our results concerning EIPs to obtain Gnedin's de Finetti-type representation result for \emph{exchangeable composition structures} (see \cite{gnedin}). 
	
	We finish this introduction with three subsections: First we describe the simplex point of view concerning the sets $\cME$ and $\cMI$, next we briefly explain the connection to ordinary (interval) graphs and then we introduce and briefly discuss the \emph{random exchangeable linear order on $\bN$}. The latter statements seem to be 'folklore'. We will relate to them a lot and therefore we think it pays to present them in a condensed form.
	
	\subsection{Choquet simplices}
	
	A metrizable Choquet simplex (just 'simplex' for short) is a metrizable compact convex set $\cM$ in which every point $x\in \cM$ can be expressed in a unique way as a mixture of extreme points. Let $\text{ex}(\cM)\subseteq\cM$ be the set of extreme points of $\cM$. Mixtures of extreme points are directed by Borel probability measures $\mu$ concentrated on $\text{ex}(\cM)$: To every $\mu$ corresponds a unique point $x_{\mu}=\int\nolimits_{\text{ex}(\cM)}y~d\mu(y)$ and the map $\mu\mapsto x_{\mu}$ is an affine surjective and bijective map from $\cM_1(\text{ex}(\cM))$ to $\cM$. We direct the reader to \cite{phelps, glasner2003ergodic, kallenberg} for an introduction and additional material concerning Choquet theory and simplices. Next we explain why and in what sense both $\cME$ and $\cME$ can be seen as simplices.
	
	Consider the already introduced set 
	$$\bIH(\bN)=\{\hH=(\hH_n)_{n\in\bN}:\hH_n\in\bIH(n)~\text{and}~\hH_n=(\hH_{n+1})_{|n}~\text{for all $n$}\}$$
	of interval hypergraphs on $\bN$. The set $\bIH(\bN)$ is a compact subset of the compact metrizable discrete product space $\prod\nolimits_{n}\bIH(n)$. Hence the space of all Borel probability measures on $\bIH(\bN)$, denoted by $\cM_1(\bIH(\bN))$, is a compact metrizable space under the topology of weak convergence. It is easily checked that $\cME\subseteq \cM_1(\bIH(\bN))$ is a compact and convex subset in that space. One can observe that $\cME$ is a simplex by use of some basic ergodic theoretic facts: Denote by $\bS_{\infty}$ the countable amenable group of finite bijections of $\bN$. One can introduce a group action from $\bS_{\infty}$ to $\PH$ such that $\cME$ is precisely the set of $\bS_{\infty}$-invariant laws on $\bIH(\bN)$:  For some $\pi\in\bS_{\infty}$ let $|\pi|:=\min\{n\in\bN:\pi(i)=i~\text{for all}~i\geq n\}$ be the size of $\pi$. Now given some $\hH=(\hH_n)_n\in\bIH(\bN)$ and some $\pi\in\bS_{\infty}$  with $m:=|\pi|$ define $\pi(\hH):=\hH'=(\hH'_n)_n$ by 
	$$\hH'_n:=\{\pi(e):e\in\hH_n\}~\text{for $n\geq m$~~~~~and}~~~~~\hH'_n:=(\hH'_{m})_{|n}~\text{for $n<m$}.$$
	It is easy to see that $\pi(\hH)\in\bIH(\bN)$ and that $\bS_{\infty}\times\bIH(\bN)\rightarrow\bIH(\bN), (\pi,\hH)\mapsto \pi(\hH)$ is a group action of $\bS_{\infty}$ on $\bIH(\bN)$ such that $\hH\mapsto \pi(\hH)$ is a homeomorphism on $\bIH(\bN)$ for every $\pi$. Now $\cME$ are precisely the $\bS_{\infty}$-invariant Borel probability measures on the compact metrizable space $\bIH(\bN)$ and so $\cME$ is a non-empty simplex (see \cite{glasner2003ergodic}, Chapter 4). Moreover, the extreme points of the convex set $\cME$ are precisely the ergodic $\bS_{\infty}$-invariant probability measures on $\PH$, where some $P\in\cME$ is called ergodic iff for every $\bS_{\infty}$-invariant event $E\subseteq \bIH(\bN)$ it holds that $P(E)\in\{0,1\}$. Introduce
	\begin{align*}
	\ex(\cME)&:=\Big\{\cL(H):~\text{$H\in\bIH(\bN)$ has a $\bS_{\infty}$-invariant ergodic law}\Big\}.
	\end{align*}
	The following is a well known fact:
	$$P\in\ex(\cME)~~\Longleftrightarrow~~P~\text{is an extreme point of the convex set}~\cME.$$
	The general theory concerning compact convex sets yields that $\ex(\cME)$ is a $G_{\delta}$-subset of $\cME$. The uniqueness of the well known ergodic decomposition can now be stated in the following explicit form: For every Borel probability measure $\mu$ on $\ex(\cME)$ the \emph{integration of $\mu$}, given by $P_{\mu}$ with
	\begin{align}\label{eq:decomposition}
		P_{\mu}(A):=\int\nolimits_{\ex(\cME)}Q(A)d\mu(Q)~~\text{for every event $A\subseteq \bIH(\bN)$,}
	\end{align}
	is the law of some exchangeable interval hypergraph, so $P_{\mu}\in\cME$ and the map 
	$$\cM_1(\ex(\cME))\rightarrow\cME,~\mu\mapsto P_{\mu}$$
	is an affine continuous bijection. The inverse of $\mu\mapsto P_{\mu}$ is given as follows: For $P\in\cME$ let $\mu^P$ be the law of the conditional distribution of $P$ given the $\bS_{\infty}$-invariant $\sigma$-field (law under $P$ itself). With this it holds that $\mu^P(\ex(\cME))=1, P_{\mu^P}=P$ and $\mu=\mu^{P_{\mu}}$ for every $\mu\in\cM_1(\ex(\cME))$. We will describe the ergodic laws not only as a set but give some insight into the intrinsic topology as well. In particular, we will see that $\ex(\cME)$ is a closed, hence compact, space and by that we identify $\cME$ as a so called \emph{Bauer simplex}. In view of Theorem \ref{thm:start} the ergodic exchangeable interval hypergraphs are precisely those that can be represented by some deterministic compact set $K$.
	
	As it is the case with the space of laws of exchangeable interval hypergraphs on $\bN$, the space $\cMI$ of laws of erased-interval processes is a simplex as well: Every $\cL((I,\eta))\in\cMI$ can be considered as a Borel probability measure on the path space $\prod\nolimits_{n\in\bN}\bIS(n)\times[n+1]$ with its product topology. Now $\cMI$ is a compact convex subset of the space of all Borel probability measures on that path space and in fact, it is a simplex: $\cMI$ is equal to the set of all \emph{Markov laws with prescribed co-transition probabilities} $\theta=(\theta^{n+1}_n,n\in\bN)$, where $\theta^{n+1}_n$ is given by 
	\begin{align}\label{eq:theta}
	\theta^{n+1}_n:\Big(\bIS(n)\times[n+1]\Big)\times&\Big(\bIS(n+1)\times[n+2]\Big)\longrightarrow[0,1],\\
	&\theta\Big((\hI_n,k_n),(\hI_{n+1},k_{n+1})\Big):=\frac{1}{n+1}1_{\{\hI_n\}}\left(\phi^{n+1}_n(\hI_{n+1},k_n)\right) \nonumber.
	\end{align}
	General theory implies that $\cMI$ is a metrizable Choquet simplex (see \cite{vershik2015equipped}) and that some $\cL((I,\eta))\in\cMI$ is an extreme point of the convex set $\cMI$ iff the terminal $\sigma$-field $\cF_{\infty}=\bigcap\nolimits_{n\in\bN}\sigma(I_m,\eta_m:m\geq n)$ generated by $(I,\eta)$ is trivial almost surely, that is every terminal event has probability either zero or one. Introduce: 
	\begin{align*}
	\ex(\cMI)&:=\Big\{\cL((I,\eta)):\text{$(I,\eta)$ is an EIP generating an a.s. trivial terminal $\sigma$-field}\Big\}.
	\end{align*}
	Like above some $P$ is an extreme point of the convex set $\cMI$ iff $P\in\ex(\cMI)$. 
	The ergodic decomposition for $\cMI$ can now be stated in the exact same form as in (\ref{eq:decomposition}) and the corresponding map $\cM_1(\ex(\cMI))\rightarrow \cMI, \mu\mapsto P_{\mu}$ is again continuous affine and bijective. We show that $\ex(\cMI)$ is compact by providing an explicit homeomorphism to a compact metric space. 
	
	\subsection{Connections to (interval) graphs} A graph $(V,E)$ is called an interval graph, if there exists a collection of intervals $I_v,v\in V$ with $I_v\subseteq (0,1)$ such that there is an edge  $\{i,j\}\in E$ iff $I_j\cap I_i\neq \emptyset$. In \cite{dhj} \emph{interval graph limits} have been studied. The upper triangle $\{(x,y):0\leq x\leq y\leq 1\}$ also appears in \cite{dhj} due to an identification of points and intervals. Nevertheless, the questions and answers concerning interval graph limits and exchangeable interval hypergraphs do not seem to overlap very much. This can be seen in the way one measures the sizes of the discrete objects: With interval hypergraphs the size equals the number of atoms, whereas in interval graphs the size equals the total number of intervals. 
	
	However, the set $\bIS(n)$ of interval systems on $[n]$ has a cardinality of precisely $2^{\binom{n}{2}}$ and thus we could choose bijections to ordinary graphs and think of erased-interval processes $(I_n,\eta_n)_{n\in\bN}$ as stochastic processes, where $(I_n)_{n\in\bN}$ is a sequence of randomly growing graphs. One possible bijection is given as follows: For $\hI\in\bIS(n)$ we define the graph $G_{\hI}$ on the set of nodes $[n]$ in which $\{a,b\}$ is an edge iff $[a,b]\in\hI$. The map $\hI\mapsto G_{\hI}$ clearly establishes a bijection from $\bIS(n)$ to the set of graphs on $[n]$. However, the operations we perform on interval systems (applying $\phi$) are (at least to us) unnatural when interpreting them as operations on graphs in the previously explained way. Therefore, we will not talk about ordinary graphs in the sequel any more. 
	
	\subsection{The exchangeable linear order} Consider the set $\bL$ of all linear orders on $\bN$. Given some $l\in\bL$ and some $n\in\bN$ let $l_{|n}$ be the restriction of $l$ to the set $[n]$. Equip $\bL$ with the $\sigma$-field generated by all these restriction maps. One defines a group action from $\bS_{\infty}$ to $\bL$ in the following way: If $l\in\bL$ and $\pi\in\bS_{\infty}$ then $\pi(l)\in\bL$ is defined by 
	$$i~\pi(l)~j~~:\Longleftrightarrow~~\pi^{-1}(i)~l~\pi^{-1}(j)~~~~\text{for all}~i,j\in\bN.$$
	 Let $L$ be a random linear order on $\bN$ with an exchangeable law, that is $\pi(L)\overset{\cD}{=}L$ for every $\pi\in\bS_{\infty}$. The law of such an object is unique, for every $n\in\bN$ the restriction $L_{|n}$ is uniform on the finite set of all possible linear orders on $[n]$. Such an exchangeable linear order $L$ naturally occurs in the context of exchangeability, but often not directly in the form of a linear order: There are other types of stochastic objects that are in some sense equivalent to an exchangeable linear order. We now introduce some notations that are needed throughout the whole paper. 
	\begin{notation}\label{notation:continous}
		For $k\in\bN$ we define
		\begin{enumerate}
			\item[-] The set $[0,1]^k_{\neq}$ consisting of all $(u_1,\dots,u_k)\in[0,1]^k$ with $u_i\neq u_j$ for all $i\neq j$. 
			\item[-] The set $[0,1]^k_{<}$ consisting of all $(u_1,\dots,u_k)\in[0,1]^k$ that are strictly increasing, so $0\leq u_1<u_2<\dots <u_{k-1}<u_k\leq 1$.
			\item[-] Given some $(u_1,\dots,u_k)\in[0,1]^k_<$ define $u_0:=-1$ and $u_{k+1}:=2$ (to avoid unpleasant case studies in some of the following definitions). 
			\item[-] For $(u_1,\dots,u_k)\in[0,1]^k_{\neq}$ let $\pi:=\pi(u_1,\dots,u_k)\in\bS_k$ be the unique permutation of $[k]$ such that $u_{\pi(1)}<\dots <u_{\pi(k)}$. Define $u_{i:k}:=u_{\pi(i)}$. In particular, $u_{0:k}=-1$ and $u_{k+1:k}=2$.
		\end{enumerate}
	\end{notation}
	
	\begin{definition}\label{def:processes}
		We define three types of stochastic processes indexed by $\bN$:
		\begin{itemize}
			\item A process $U=(U_i)_{i\in\bN}$ such that
			\begin{itemize}
				\item $U_1,U_2,\dots$ are independent identically distributed,
				\item each $U_i$ is uniform on the unit interval, $U_i\sim\unif([0,1])$,
			\end{itemize}
			is called an \emph{$U$-process}.
			\item A process $S=(S_n)_{n\in\bN}$ such that  for each $n$
			\begin{itemize}
				\item $S_n$ is a uniform random permutation of $[n]$, so $S_n\sim\unif(\bS_n)$,
				\item the one-line notation of $S_n$ is almost surely obtained from $S_{n+1}$ by erasing '$n+1$' in the one-line notation of $S_{n+1}$
			\end{itemize}
			is called a \emph{permutation process}.
			\item A process $\eta=(\eta_n)_{n\in\bN}$ such that 
			\begin{itemize}
				\item $\eta_1,\eta_2,\dots$ are independent,
				\item $\eta_n$ is uniformly distributed on the finite set $[n+1]=\{1,\dots,n+1\}$ for each $n$, 
			\end{itemize}
			is called an \emph{eraser process}.
		\end{itemize}
	\end{definition}

	If $(I,\eta)$ is an erased-interval process, then $\eta$ is an eraser process. Next we are going to explain in what sense the four introduced objects -- exchangeable linear order $L$, $U$-process, permutation process $S$, eraser process $\eta$ -- can be considered to be equivalent, that is given any one of the four types of stochastic objects one can pass to any other in an almost surely defined functional way without loosing probabilistic information:
	
	\underline{$U\rightarrow L$:}~~Given some $U$-process $U=(U_i)_{i\in\bN}$ one can define a random linear order $L$ on $\bN$ by $$i~L~j:\Longleftrightarrow U_i<U_j.$$ This random linear order is exchangeable by the exchangeability of $U$. 
	
	\underline{$L\rightarrow U$:}~~Given some exchangeable linear order $L$ and some $i\in\bN$, the previously introduced explicit construction of $L$ directly yields that the limit $$U_i=\lim\limits_{n\rightarrow\infty}\frac{\#\{k\in[n]:k~L~i\}}{n}$$ exists almost surely for all $i$ and yields an $U$-process $U=(U_i)_{i\in\bN}$. 
		
	\underline{$U\rightarrow S$:}~~Given some $U$-process and some $n\in\bN$ define the random permutation
	$$S_n:=\pi(U_1,\dots,U_n)$$ of $[n]$ that arranges the first $n$ $U$-values in increasing order. $S=(S_n)_{n\in\bN}$ is a permutation process, which follows from the exchangeability of $U$ and from the algorithmic construction. 
	
	\underline{$S\rightarrow U$:}~~Given some permutation process $S$ and some $i\in\bN$ the limit
	$$U_i=\lim\limits_{n\rightarrow\infty}\frac{S_n^{-1}(i)}{n}$$ exists almost surely and these limits form an $U$-process $U=(U_i)_{i\in\bN}$. This can be seen by representing $S$ as the permutation process corresponding to some $U$-process $U'$. One obtains $S_n^{-1}(i)=\#\{j\in[n]:U'_j\leq U'_i\}$. The strong law of large numbers yields $U'_i=\lim_nn^{-1}S_n^{-1}(i)$ almost surely.
	
	\underline{$S\rightarrow \eta$:}~~If $S$ is a permutation process then $\eta=(\eta_n)_{n\in\bN}$ with 
	$$\eta_n:=S_{n+1}^{-1}(n+1)$$ is an eraser process. This is due to the fact that $S_n$ is a uniform on $\bS_n$ and $\#\bS_{n+1}=(n+1)!=(n+1)\cdot \#\bS_{n}$.
	
	\underline{$\eta\rightarrow S$:}~~We introduce, for every $n\geq 2$, the bijection 
	\begin{equation}\label{eq:algperm}
	b_n:[2]\times[3]\times\cdots\times[n]\rightarrow\bS_n,
	\end{equation}
	where $b_n$ is defined inductively: $b_1$ is the unique permutation of $[1]$ and the one-line notation of $\pi=b_n(i_1,\dots,i_{n-1})$ is obtained from the one-line notation of $\pi'=b_{n-1}(i_1,\dots,i_{n-2})$ by placing '$n$' in the $i_{n-1}$-th gap of $\pi'=(~\framebox(5,5){}~\pi'(1)~\framebox(5,5){}~\pi'(2)~\framebox(5,5){}~\cdots~\framebox(5,5){}~\pi'(n-1)~\framebox(5,5){}~)$. Now given some eraser process $\eta$ we define $S_1:=(1)$ and for $n\geq 2$
	$$S_n:=b_n(\eta_1,\dots,\eta_{n-1}).$$ 
	$S=(S_n)_{n\in\bN}$ is a permutation process.

	One easily sees that $U\rightarrow L$ and $L\rightarrow U'$ implies $U=U'$ almost surely. The same holds for all other constructions: $A\rightarrow B$ and $B\rightarrow A'$ implies $A=A'$ almost surely in all possible situations defined above. One can connect the previously defined constructions to pass from any of the four objects to any other in an almost surely uniquely defined way. These can be expressed explicit:
	
	\underline{$L\rightarrow S$:}~~ For any $n\in\bN$ there is a unique random bijection $S_n$ of $[n]$ such that $iLj\Leftrightarrow S_n^{-1}(i)<S_n^{-1}(j)$ for all $i,j\in[n]$. The process $S=(S_n)_{n\in\bN}$ is a permutation process and is exactly the result of $L\rightarrow U\rightarrow S$.  
	
	\underline{$S\rightarrow L$:}~~ For any $i,j\in\bN$ the define random relation $L$ by $iLj:\Leftrightarrow S_n^{-1}(i)<S_n^{-1}(j)$ with $n=\max\{i,j\}$. One can show that $L$ is the exchangeable order and is exactly the result of $S\rightarrow U\rightarrow L$.  
		
	\underline{$U\rightarrow \eta$:} For $k\in\bN$ let $\eta_k$ be the relative rank of $U_{k+1}$ in the first $k+1$ $U$-values, so 
	$$\eta_k=\#\{i\in[k+1]:U_i\leq U_{k+1}\}.$$
	This $\eta$ is the result of $U\rightarrow S\rightarrow \eta$.
	
	\underline{$L\rightarrow \eta$:} For $k\in\bN$ let 
	$$\eta_k:=1+\#\{i\in[k]:i L k+1\}.$$
	This $\eta$ is the result of $L\rightarrow U\rightarrow \eta$.
	
	The missing two relations $\eta\rightarrow U$ and $\eta\rightarrow L$ are best understand by passing to $S$ first. If one starts with any of the four objects under consideration, there are almost surely uniquely defined objects of the other three types. We will relate to them as \emph{corresponding objects}. Fig. \ref{fig:eraserprocess} shows the first steps of some realization of a corresponding triple $(U,S,\eta)$. 
	
	\begin{wrapfigure}{r}{0.4\textwidth}
		\centering
		\includegraphics[width=0.4\textwidth]{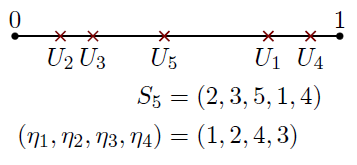}
		\caption{\label{fig:eraserprocess}}
	\end{wrapfigure}

	In particular, for any erased-interval process $(I,\eta)$ there are an $U$-process and a permutation process both corresponding to the eraser process $\eta$ and thus defined on the same probability space as $(I,\eta)$. These processes will play an important role in our representation result, since the corresponding $U$-process serves as the randomization used to \emph{sample from infinity} and the permutation process $S$ is used to pass from $(I,\eta)$ to an exchangeable interval hypergraph. 

	\section{main results}\label{sec:mainresults}
	
	Our first main theorem will be the characterization of erased-interval processes. At first we introduce a compact metric space that turns out to be homeomorphic to the space of ergodic EIPs, that is $\ex(\cMI)$. The elements of this space are \emph{limits of scaled interval systems as $n\rightarrow\infty$}. We need to recall some topological definitions: Given any metric space $(M,d)$ we introduce 
	$$\cK(M)=\{\text{all non-empty compact subsets of $M$}\}.$$ 
	On $\cK(M)$ we will consider the \emph{Hausdorff distance} defined by 
	\begin{equation}\label{def:hausdorffmetric}
	\dha\left(K_1,K_2\right):=\max\{\max\{d(x,K_2):x\in K_1\},\max\{d(x,K_1):x\in K_2\}\}.
	\end{equation}
	A well known fact is that if $(M,d)$ is a compact metric space, then so is $(\cK(M),\dha)$ (see \cite[Chapter 7]{burago}). If we talk about random compact sets, we always mean random variables taking values in the space $\cK(M)$ equipped with the Borel $\sigma$-field corresponding to $\dha$. We need the following characterization of convergence in $(\cK(M),\dha)$, see \cite[Exercise 7.3.4]{burago}:
	\begin{lemma}\label{lemma:convhausdorff}
		Let $(M,d)$ be a compact metric space and $K_n,K\in\cK(M)$ be such that $\dha(K_n,K)\rightarrow 0$. Then for every $x\in K$ there exists a sequence $(x_n)$ such that $x_n\in K_n$ for every $n$ and $d(x_n,x)\rightarrow 0$. If $(x_n)$ is a sequence with $x_n\in K_n$ for every $n$ and $d(x_n,x)\rightarrow 0$ for some $x\in M$, then $x\in K$. 
	\end{lemma}
	
	A major point for the intuition behind our constructions is that one can identify any open subinterval of the open unit interval $(0,1)$, which is a \emph{set} of the form $(x,y)=\{z\in(0,1):x<z<y\}$ with the \emph{point} in $\bR^2$ whose coordinates are given by the end points of that interval, $(x,y)\in\bR^2$ (note the present overloading of symbols). Consider the \emph{triangle} $\recto{}=\big\{(x,y)\in\bR^2:0\leq x\leq y\leq 1\big\}$ introduced before and let
	$$\diago{}:=\big\{(x,x)\in\bR^2:0\leq x\leq 1\big\}$$
	be the diagonal line from $(0,0)$ to $(1,1)$. Consider the metric induced by the $1$-norm on $\bR^2$, so $d((x_1,y_1),(x_2,y_2)):=|x_1-x_2|+|y_1-y_2|$. Define
	\begin{equation}
	\bISi:=\big\{K\subseteq\recto{}:~\text{$K$ is compact and}~\diago{}\subseteq K\big\}.
	\end{equation}
	The following lemma is very easy to prove, but nevertheless of great importance later on.
	\begin{lemma}
		$\bISi$ is a closed subset of $\cK(\recto{})$ and $(\bISi,\dha)$ is a compact metric space. In particular, $(\bISi,\dha)$ is complete. 
	\end{lemma}
	\begin{proof}
		We only need to prove the first statement, since $(\cK(\recto{}),\dha)$ is known to be a compact metric space and closed subspaces of compact metric spaces are compact. Furthermore, every compact metric space is complete. So let $K_n\in \bISi$ and $K\in\cK(\recto{})$ be such that $\dha(K_n,K)\rightarrow 0$. Since $(x,x)\in K_n$ for every $x\in[0,1]$ and every $n\in\bN$, Lemma \ref{lemma:convhausdorff} yields that $(x,x)\in K$. Hence $\diago{}\subseteq K$ and $K\in\bISi$. 
	\end{proof}
	The space $(\bISi,\dha)$ will turn out to be homeomorphic to $\ex(\cMI)$. As the notation indicates, $\bISi$ can be seen, in various ways, as the analogue for interval systems $\bIS(n)$ as $n\rightarrow\infty$. Our main theorem says that \emph{every ergodic erased-interval process can be obtained by sampling from some unique $K\in\bISi$, even in a homeomorphic way}. We will now present the map that describes this 'sampling from infinity':

	Let $k\in\bN$ and define 
	\begin{align*}
	\phi^{\infty}_k:&\bIS(\infty)\times [0,1]^k_<\longrightarrow\bIS(k),\\
	&\phi^{\infty}_k(K,u_1,\dots,u_k):=\Bigg\{[a,b]:\begin{aligned}
	&1\leq a\leq b\leq k~\text{s.t. exists $(x,y)\in K$ with}\\
	&u_{a-1}<x<u_a\leq u_b<y<u_{b+1}\\
	\end{aligned}\Bigg\}\\
	&~~~~~~~~~~~~~~~~~~~~~~~~~~~~\bigcup\Big\{\{j\}:j\in[k]\Big\}\bigcup\Big\{\emptyset\Big\}.
	\end{align*}
	
	So let $K\in\bISi, (u_1,\dots,u_k)\in [0,1]^k_<$ and $1\leq a<b\leq k$. One directly obtains the following very useful description of $\phi^{\infty}_k$:
	\begin{equation}\label{eq:samplingREP}
	[a,b]\in\phi^{\infty}_k(K,u_1,\dots,u_k)~\Longleftrightarrow~K\cap (u_{a-1},u_a)\times(u_{b},u_{b+1})\neq\emptyset.
	\end{equation}
	
	One may wish to take a look at Fig.\ref{fig:phiinfty}.

\begin{figure}[h!]
	\centering
	\includegraphics[width=1.05\textwidth]{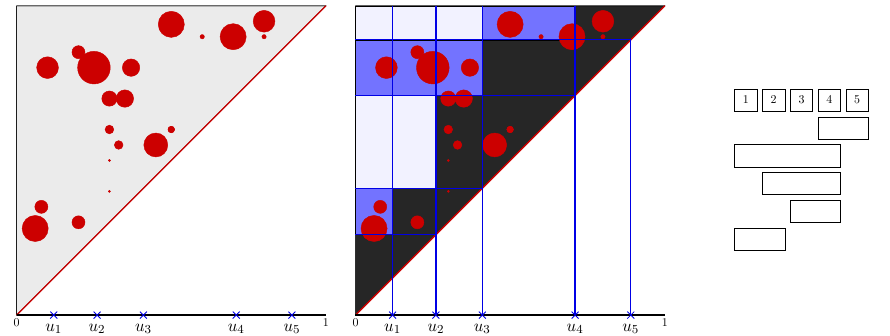}
	\caption{On the left a compact set $\color{red!80!black}K\color{black}\in\bISi$ and some points $(\color{blue!90!black}u_1\color{black},\color{blue!90!black}u_2\color{black},\color{blue!90!black}u_3\color{black},\color{blue!90!black}u_4\color{black},\color{blue!90!black}u_5\color{black})\in [0,1]^5_<$. These $5$ points divide the upper triangle into parts of which $\binom{5}{2}$ are relevant to the definition of $\phi^{\infty}_5(\color{red!80!black}K\color{black},\color{blue!90!black}u_1\color{black},\color{blue!90!black}u_2\color{black},\color{blue!90!black}u_3\color{black},\color{blue!90!black}u_4\color{black},\color{blue!90!black}u_5\color{black})\in\bIS(5)$. The latter can be seen on the right: an interval $[a,b]$ with $1\leq a<b\leq 5$ is present in the induced interval system iff the compact set $\color{red!80!black}K\color{black}$ intersects the the open rectangle $\color{blue!90!black}(u_{a-1},u_a)\times(u_{b},u_{b+1})\color{black}$, here again $\color{blue!90!black}u_0\color{black}:=-1$ and $\color{blue!90!black}u_6\color{black}:=2$.}
	\label{fig:phiinfty}
\end{figure}

	\begin{remark}
		Let $(u_1,\dots,u_k)\in[0,1]^k_<$. We used the conventions $u_0:=-1$ and $u_{k+1}:=2$ because we want points $(x,y)\in K\in\bISi$ with $x=0$ or $y=1$ to eventually have an effect on $\phi^{\infty}_k(K,u_1,\dots,u_k)$. Since we work with open rectangles (see (\ref{eq:samplingREP})) the choices $u_0:=0$ and $u_{k+1}:=1$ would have failed to achieve this. 
	\end{remark}
	
	We will prove that the map $\phi^{\infty}_k$ is measurable for every $k$ with respect to the Borel $\sigma$-field on $\bISi\times[0,1]^k_<$. Thus one can plug in random elements and obtain $\bIS(k)$-valued random elements. In particular, we will plug in the order statistics $(U_{1:k},\dots,U_{k:k})$ obtained from the $U$-processes $(U_i)_{i\in\bN}$ corresponding to an eraser process $\eta$ which stems from an erased-interval process $(I,\eta)$. 
	
	Let $(U_i)_{i\in\bN}$ be an $U$-process and let $\eta$ be the eraser process corresponding to $U$. We will prove that for every $K\in\bISi$ the stochastic process 
	\begin{equation}\label{def:sampledEIP}
	\Big(\phi^{\infty}_n(K,U_{1:n},\dots,U_{n:n}),\eta_n\Big)_{n\in\bN}
	\end{equation}
	is an ergodic erased-interval process and that \emph{every} ergodic erased-interval process is of this form; \emph{not only in law but almost surely}. Denote the law of the process in (\ref{def:sampledEIP}) with $\law(K)$. So in particular, we will show that $\law(K)\in\ex(\cMI)$ for every $K\in\bISi$. To prove this almost sure representation we need to explain how to obtain an appropriate compact subset $K\in\bISi$ when given an erased-interval process $(I,\eta)$. This desired interval system is obtained by scaling  $I_n$ and then letting $n\rightarrow\infty$. We now introduce this scaling procedure. For $n\in\bN$ and $\hI\in\bIS(n)$ let
	\begin{equation}
		n^{-1}\hI:=\Bigg\{\Big(\frac{a-1}{n},\frac{b}{n}\Big):\emptyset\neq[a,b]\in\hI\Bigg\}\cup\diago{}.
	\end{equation}
	In particular, $n^{-1}\hI\subset \recto{}$ and $\diago{}\subseteq n^{-1}\hI$ be definition and so $n^{-1}\hI\in\bISi$, since $n^{-1}\hI$ is obviously compact. See Fig. \ref{fig:scaling} for a visualization.
	\begin{figure}[h!]
		\centering
		\includegraphics[width=0.75\textwidth]{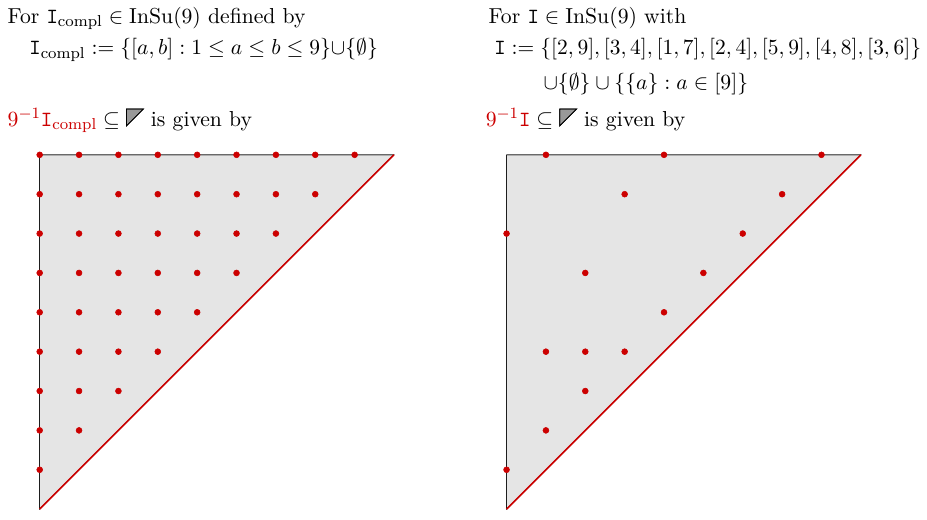}
		\caption{\label{fig:scaling}}
	\end{figure}
	We will prove that $n^{-1}I_n$ converges almost surely in the space $(\bISi,\dha)$ for every erased-interval process $(I_n,\eta_n)_{n\in\bN}$. 
	
	We are now ready to state our main theorems.
	\begin{theorem}\label{thm:mainEIP}
		For every $K\in\bISi$ one has $\law(K)\in\ex(\cMI)$ and the map $\bISi\rightarrow\ex(\cMI),K\mapsto \law(K)$ is a homeomorphism. One has the following almost sure representation: Let $(I,\eta)=(I_n,\eta_n)_{n\in\bN}$ be an erased-interval process. Then $n^{-1}I_n$ converges almost surely as $n\rightarrow\infty$ towards some $\bISi$-valued random variable $I_{\infty}$. Let $U=(U_i)_{i\in\bN}$ be the $U$-process corresponding to $\eta$. Then $I_{\infty}$ and $U$ are independent and one has the equality of processes
		\begin{equation*}
		(I_n,\eta_n)_{n\in\bN}~=~\Big(\phi^{\infty}_n(I_{\infty},U_{1:n},\dots,U_{n:n}),\eta_n\Big)_{n\in\bN}~~\text{almost surely.}
		\end{equation*}
		In particular, for every erased-interval process $(I,\eta)$ the conditional law of $(I,\eta)$ given the terminal $\sigma$-field $\cF_{\infty}$ is $\law(I_{\infty})$ almost surely and $I_{\infty}$ generates $\cF_{\infty}$ almost surely.
	\end{theorem}
	
	\begin{theorem}\label{thm:mainMAP}
		Let $(I,\eta)=(I_n,\eta_n)_{n\in\bN}$ be an erased-interval process and let $(S_n)_{n\in\bN}$ be the permutation process corresponding to $\eta$. Let $(U_i)_{i\in\bN}$ be the $U$-process corresponding to $\eta$ and let $I_{\infty}=\lim\nolimits_{n\rightarrow\infty}n^{-1}I_n$. Define $H_n:=S_n(I_n)$. Then it holds that 
		\begin{equation}\label{eq:defehy}
			H_n=\Big\{\{j\in[n]:x<U_j<y\}:(x,y)\in I_{\infty}\Big\}\cup\{\{j\}:j\in[n]\}\cup\{\emptyset\}~\text{a.s. for every $n$}
		\end{equation}
		and $H=(H_n)_{n\in\bN}$ is an exchangeable interval hypergraph on $\bN$. The map 
		\begin{equation*}
		\cMI\longrightarrow\cME,~\cL((I,\eta))\longmapsto \cL(H)
		\end{equation*}
		is a continuous affine surjection.
	\end{theorem}  
	
	We will shortly state some corollaries that follow easily from the previous two theorems: Let $K_1$ and $K_2$ be convex sets with extreme points $\text{ex}(K_1),\text{ex}(K_2)$ and let $f:K_1\rightarrow K_2$ be an affine surjective map. Then it holds that $f^{-1}(\text{ex}(K_2))\subseteq \text{ex}(K_1)$. This can be applied to $K_1=\cMI, K_2=\cME$ and $f$ as in Theorem \ref{thm:mainMAP}. One easily sees that the map $f$ in this situation maps extreme points to extreme points: every exchangeable interval hypergraph $H$ that is constructed as in (\ref{eq:defehy}) with some \emph{deterministic} $K=I_{\infty}\in\bISi$ is ergodic, due to the Hewitt-Savage zero-one law. Hence $f(\text{ex}(K_1))=\ex(K_2)$. One can summarize these considerations to the following
	
	\begin{corollary}
		Let $K\in\bISi$ and $U=(U_i)_{i\in\bN}$ be an $U$-process. Then the process 
		\begin{equation}\label{eq:defehyi}
			\Bigg(\Bigg\{\Big\{j\in[n]:x<U_j<y\Big\}:(x,y)\in K\Bigg\}\cup\Big\{\{j\}:j\in[n]\Big\}\cup\{\emptyset\}\Bigg)_{n\in\bN}
		\end{equation}
		is an ergodic exchangeable interval hypergraph on $\bN$ and the law of \emph{every} ergodic exchangeable interval hypergraph on $\bN$ can be expressed in this form. Denote the law of (\ref{eq:defehyi}) by $\law^{\text{ih}}(K)$. The map $\bISi\rightarrow\ex(\cME),K\mapsto \law^{\text{ih}}(K)$ is surjective and continuous.
	\end{corollary}
	
	The next corollary is about the structure of the spaces $\cMI$ and $\cME$ as simplices:
	
	\begin{corollary}
		The simplex $\cMI$ is a Bauer simplex affinely homeomorphic to the simplex of all Borel probability measures on $\bISi$ equipped with the topology of weak convergence. The simplex $\cME$ is also a Bauer simplex: its extreme points are a continuous image of the compact space $\bISi$. 
	\end{corollary}

	\subsection{Poly-adic filtrations} Now we shortly explain how one can easily deduce a statement concerning certain poly-adic backward filtrations generated by erased-interval processes and explain why all singletons $\{j\}$ are assumed to be part of any interval hypergraph. For an introduction to poly-adic filtrations and further references we refer the reader to \cite{Leuridan}. One should emphasize that the properties concerning (backwards) filtrations we are going to state are properties concerning \emph{filtered probability spaces}, so they are, in general, not stable under a change of measure. Given a probability space $(\Omega,\cA,\bP)$ and sub-$\sigma$-fields $\cB,\cC\subseteq\cA$, we say that $\cB\subseteq \cC$ holds almost surely iff for every $B\in\cB$ there is a $C\in\cC$ such that $\bP(B\Delta C)=0$. Consequently, $\cB=\cC$ almost surely iff both $\cB\subseteq\cC$ and $\cC\subseteq\cB$ hold almost surely.  
	
	Consider the backward filtration $\cF$ generated by an erased-interval process $(I,\eta)$, so $\cF=(\cF_n)_{n\in\bN}$ with 
	$$\cF_n=\sigma(I_m,\eta_m:m\geq n).$$
	Since $I_{n-1}=\phi^n_{n-1}(I_n,\eta_{n-1})$ holds almost surely for every $n\geq 2$, one has that 
	$$\cF_{n-1}=\cF_n\vee\sigma(\eta_{n-1})~~\text{almost surely for every}~n\geq 2.$$
	By definition, $\eta_{n-1}$ is independent of $\cF_n$ and uniformly distributed on the finite set $\{1,\dots,n\}$. The process $\eta$ is called a process of \emph{local innovations} for $\cF$ and the backward filtration $\cF$ is an example of a  \emph{poly-adic} (backward) filtration. Inductively applying the above almost sure equality of $\sigma$-fields yields 
	$$\cF_k=\sigma(\eta_k,\eta_{k+1},\dots,\eta_{n-1})\vee\cF_{n}~~\text{almost surely for every}~1\leq k<n.$$
	Now $\sigma(\eta_k,\eta_{k+1},\dots,\eta_{n-1})\vee\cF_{n}=\sigma(\eta_m:m\geq k)\vee\cF_n$ for all $1\leq k<n$ holds by definition of $\cF_n$. Via this one obtains
	$$\cF_k=\bigcap\limits_{n>k}[\sigma(\eta_m:m\geq k)\vee \cF_n]~~\text{almost surely for every $k\in\bN$.}$$
	Since $\sigma(\eta_m:m\geq k)$ does not depend on $n$, one may wonder if one can \emph{interchange the order of taking the supremum and taking the intersection} on the right hand side in the last equation. This is \emph{not always} allowed: In \cite{weizsacker1983exchanging} one can find a treatment of such questions in a very general setting. However, Theorem \ref{thm:mainEIP} shows that this interchange \emph{is allowed} in our concrete situation: For this one only needs to observe that 
	$$\sigma(\eta_k,\eta_{k+1},\eta_{k+2},\dots)=\sigma(U_{1:k},\dots,U_{k:k},U_{k+1},U_{k+2},\dots)~~\text{almost surely for every}~k\geq 1,$$
	where $U$ is the $U$-process corresponding to $\eta$. Of course, $I_{\infty}$ is $\cF_{\infty}$-measurable. The representation $I_k=\phi^{\infty}_k(I_{\infty},U_{1:k},\dots,U_{k:k})$ almost surely thus yields that 
	$$\sigma(I_k)\subseteq \sigma(U_{1:k},\dots,U_{k:k})\vee\sigma(I_{\infty})\subseteq \sigma(\eta_k,\eta_{k+1},\dots)\vee\cF_{\infty}~\text{almost surely for every}~k\geq 1.$$
	Hence one obtains  
	$$\cF_k=\sigma(\eta_m:m\geq k)\vee\bigcap\limits_{n>k}\cF_n~\text{almost surely for every}~k\geq 1.$$
	In particular, if $\cF_{\infty}$ is Kolmogorovian, that is if $\cF_{\infty}$ is a.s. trivial, then $\cF$ is almost surely generated by $\eta$ and so it is of product type. This yields
	
	\begin{corollary}
		Let $(I,\eta)=(I_n,\eta_n)_{n\in\bN}$ be an erased-interval process and let $\cF$ be the backwards filtration generated by $(I,\eta)$. Then $\cF$ is of product type iff it is Kolmogorovian and in particular, $\eta$ generates $\cF$ almost surely.  
	\end{corollary}

	In Section \ref{sec:connections} we explain in what sense every infinite labeled R\'emy bridge can be seen as an erased-interval process. The above statement concerning the filtrations was already formulated in \cite[Lemma 5.3.]{egw2}, but the proof they give contains errors (see the Annex in \cite{Leuridan}). However, our result shows that the lemma formulated in \cite{egw2} is correct.

	In \cite{gerstenberg} and \cite{laurent2016filtrations} different erased-type processes and their backward filtrations have been analyzed: \emph{(general) erased-word processes}. A general erased-word process over a finite alphabet $\Sigma$ is a stochastic process $(W_n,\eta_n)_{n\in\bN}$ that is almost like an erased-interval process, but with the following differences: $W_n=(W_{n,1},\dots,W_{n,n})$ is a random word of length $n$ over the alphabet $\Sigma$ and $W_{n-1}$ is obtained by erasing the $\eta_{n-1}$-th letter from $W_n$. In \cite{gerstenberg} it was shown that
	
	\begin{theorem*}
		The backward filtration generated by a general erased-word process over some finite alphabet $\Sigma$ is of product-type iff it is Kolmogorovian, \emph{but it is in this case not always generated by $\eta$}. 
	\end{theorem*}

	If one had defined interval hypergraphs such that singleton sets may or may not be part of the edge sets, then some erased-interval process $(I_n,\eta_n)_{n\in\bN}$ would have included non-trivial erased-word processes over the alphabet $\Sigma=\{0,1\}$: $W_{n,i}=1:\Leftrightarrow \{i\}\in I_n.$ Not only would the description of the ergodic laws have become a more challenging task, but an almost sure functional representation in the spirit of Theorem \ref{thm:mainEIP} would not have been possible and our method of proof would have failed. Some applications of Laurent's results concerning erased-word processes can also be found in \cite{Leuridan}.

	\section{proofs}\label{sec:proof}
	
	Now we will prove our main theorems, most of the effort lies in the proof of Theorem~\ref{thm:mainEIP}. We will first gather some lemmas and finally put them together. We need to introduce some notation, which are the finite analogues of the ones introduced in Notation \ref{notation:continous}:
	\begin{notation}\label{notation:discrete}
		Let $n\in\bN$ and $1\leq k\leq n$.
		\begin{enumerate}
			\item[-] $[k:n]$ is the set of all vectors $\vec{j}=(j_1,\dots,j_k)\in[n]^k$ that are strictly increasing, so $1\leq j_1<j_2<\dots<j_k\leq n$.
			\item[-] For $\vec{j}=(j_1,\dots,j_k)\in[k:n]$ we define $j_{0}:=-n$ and $j_{k+1}:=2n$ (to avoid unpleasant case studies in some of the following definitions).
			\item[-] Given some permutation $\pi\in\bS_n$ and some $1\leq k\leq n$ we define $\vj^{\pi}_{k}\in[k:n]$ to be the increasing enumeration of the set $\pi^{-1}([k])\subseteq [n]$. So $\vj^{\pi}_{k}$ is the unique vector $\vj^{\pi}_k=(j^{\pi}_{k,1},\dots,j^{\pi}_{k,k})\in[k:n]$ with $\{j^{\pi}_{k,1},\dots,j^{\pi}_{k,k}\}=\pi^{-1}([k])$. In particular, $j^{\pi}_{k,0}=-n$ and $j^{\pi}_{k,k+1}=2n$. 
			\item[-] Given some $\vj=(j_1,\dots,j_k)\in[k:n]$ we define $$n^{-1}\vj:=\left(\frac{2j_1-1}{2n},\dots,\frac{2j_k-1}{2n}\right)\in[0,1]^k_<.$$
		\end{enumerate} 
	\end{notation}
	
	Up to now, the restriction map $\phi^{n+1}_n$ has only been considered for successive numbers $(n,n+1)$. We will now present the multi-step restriction functions $\phi^{n}_k, 1\leq k\leq n$ and show that $\phi^{\infty}_k$ are, in some sense, the limiting analogies for fixed $k$ with $n\rightarrow\infty$. 
	
	Let $\hI_{n}\in\bIS(n)$ and $(i_{k},i_{k+1},\dots,i_{n-1})\in [k+1]\times[k+2]\times\cdots\times[n]$ be some sequence of erasers. Define inductively $\hI_{m}:=\phi^{m+1}_{m}(\hI_{m+1},i_m)$ for $m=n-1,\dots,k$. The resulting $\hI_k\in\bIS(k)$ does not depend on the full information contained in the sequence of erasers $(i_{k},\dots,i_{n-1})$ as one can interchange orders of erasing in certain senses and obtain the same result. The relevant information contained in $(i_{k},\dots,i_{n-1})$ is described by a vector from $[k:n]$: extend the eraser vector to some vector $(\star,\dots,\star,i_{k},\dots,i_{n-1})\in[2]\times\cdots\times[n]$ and use this vector to define a permutation $\pi\in\bS_{n}$ via $\pi:=b_n(\star,\dots,\star,i_{k},\dots,i_{n-1})$ (see (\ref{eq:algperm})). Now $\hI_k$ only depends on $\hI_{n}$ and $\vj^{\pi}_k\in[k:n]$. This is well defined, since $\vj^{\pi}_k$ does not depend on the choices of $\star$ that were used to produce $\pi$. This functional dependences are given by the following definition:
	\begin{definition}\label{def:finitesampling}
		For $n\in\bN$ and $1\leq k\leq n$ let $\vj=(j_1,\dots,j_k)\in[k:n]$. Define $\phi^n_k:\bIS(n)\times[k:n]\longrightarrow \bIS(k)$ via
		$$\phi^{n}_k(\hI,\vj):=\Bigg\{[a,b]:\begin{aligned}
		&1\leq a\leq b\leq k~\text{s.t. exists $[A,B]\in \hI$ with}\\
		&j_{a-1}<A\leq j_a\leq j_b\leq B<j_{b+1}\\
		\end{aligned}\Bigg\}\cup\Big\{\emptyset\Big\}.$$
	\end{definition}
	
	The overloading of symbols when dealing with open intervals $(x,y)$ and points $(x,y)$ in two dimensions can be carried out for finite interval systems $\hI\in\bIS(n)$ as well. Given some non-empty interval $[a,b]\in \hI$ we can map $[a,b]$ to the point $(a,b)\in[n]\times[n]$ and via this we can interpret each $\hI\in\bIS(n)$ as a subset $\hI\subset [n]\times[n]$ (ignoring the empty set $\emptyset\in\hI$). With this in mind, we can give a description of the map $\phi^n_k$ that is in direct analogy with the one given for $\phi^{\infty}_k$ in (\ref{eq:samplingREP}): Let $1\leq k\leq n, \vj=(j_1,\dots,j_k)\in[k:n], \hI\in\bIS(n)$ and some $1\leq a<b\leq k$. Then it holds that 
	\begin{equation}\label{eq:samplingREPfin}
	[a,b]\in\phi^n_k(\hI,\vj)~\Longleftrightarrow~\hI\cap [j_{a-1}+1,j_a]\times[j_{b},j_{b+1}-1]\neq \emptyset.
	\end{equation}
	See Fig. \ref{fig:finitesampling}.
	
	\begin{figure}
		\centering
		\includegraphics[width=0.9\textwidth]{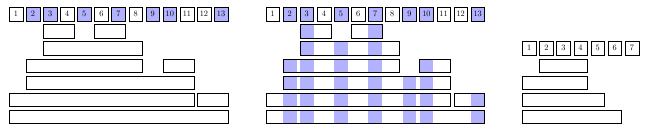}
		\caption{On the left some interval system $\hI\in\bIS(13)$. Highlighted in blue: the vector $\vj=(2,3,5,7,9,10,13)\in[7:13]$. On the right the interval system $\phi^{13}_7(\hI,\vj)$.}
		\label{fig:finitesampling}
	\end{figure}

	We notice that we have defined $\phi^{n+1}_n$ in two ways: at first in Section $1$ as a function $$\bIS(n+1)\times[n+1]\rightarrow\bIS(n)$$ and then in Definition \ref{def:finitesampling} as a function $$\bIS(n+1)\times[n:n+1]\rightarrow\bIS(n).$$ This 'overloading' of the function symbol $\phi^{n+1}_n$ is justified by noticing that $$[n+1]\rightarrow[n:n+1],~~k\mapsto (1,\dots,k-1,k+1,\dots,n+1)$$ is bijective and it holds that 
	\begin{equation}\label{eq:samedef}
		\phi^{n+1}_n(\hI,k)=\phi^{n+1}_n(\hI,(1,\dots,k-1,k+1,\dots,n+1))
	\end{equation}
	for any $n, \hI\in\bIS(n+1)$ and $k\in[n+1]$. Let $\pi\in\bS_n$ and $k\leq n$. 
	If one deletes $k+1,\dots,n$ from the one-line notation of $\pi$, one obtains the one-line notation of some permutation of $[k]$ which we will call $\pi_{|k}\in\bS_k$. This is consistent with building a permutation from a sequence of erasers: Let $(i_1,\dots,i_{n-1})\in[2]\times\cdots \times[n]$ be a sequence of erasers and $\pi=b_n(i_1,\dots,i_{n-1})$, then $\pi_{|k}=b_k(i_1,\dots,i_{k-1})$. The next lemma shows that $\phi^n_k$ really describes the multi-step deletion operations as claimed above and gives some further algorithmic properties. 
	\begin{lemma}\label{lemma:algorithmic}
		Let $n\in\bN$ and $\hI\in\bIS(n)$.
		\begin{enumerate}
			\item[a)] If $1\leq k\leq m\leq n$ and $\vec{j}=(j_1,\dots,j_m)\in[m:n], \vec{h}=(h_1,\dots,h_k)\in[k:m]$ then 
				$$\phi^m_k(\phi^n_m(\hI,\vec{j}),\vec{h})=\phi^n_k(\hI,(j_{h_1},\dots,j_{h_k})).$$
		\end{enumerate}
		Let $\pi\in\bS_n$. For $1\leq k\leq n-1$ let $i_k:=(\pi_{|k+1})^{-1}(k+1)\in[k+1]$. Define $(\hI_n,\hI_{n-1},\dots,\hI_1)$ inductively by $\hI_n:=\hI$ and $\hI_{k-1}:=\phi^{k}_{k-1}(\hI_k,i_{k-1})$ for $2\leq k\leq n$. Then for every $1\leq k\leq n$:
		\begin{enumerate}
			\item[b)] If $\vj^{\pi}_k$ is the enumeration of the set $\pi^{-1}([k])$ it holds that $\hI_k=\phi^n_k(\hI,\vj^{\pi}_k)$. 
			\item[c)] It holds that $\hI_k=\pi_{|k}^{-1}(\pi(\hI)_{|k})$, where $\pi(\hI)$ and $\hI_{|k}$ are the operations introduced for interval hypergraphs in Section~\ref{sec:INTRO}.
		\end{enumerate}
	\end{lemma}
	\begin{proof}
		a):~Let $1\leq a<b\leq k$. Equation (\ref{eq:samplingREPfin}) yields 
		\begin{align*}
			[a,b]\in \phi^m_k(\phi^n_m(\hI,\vec{j}),\vec{h})&\Longleftrightarrow \phi^n_m(\hI,\vec{j})\cap [h_{a-1}+1,h_a]\times[h_b,h_{b+1}-1]\neq \emptyset.\\
			&\Longleftrightarrow \exists [a',b']\subseteq [m]: \hI\cap [j_{a'-1}+1,j_{a'}]\times[j_{b'},j_{b'+1}-1]\neq\emptyset\\
			&~~~~~~~~~~~~~\text{and}~h_{a-1}+1\leq a'\leq h_a<h_b\leq b'\leq h_{b+1}-1\\
			&\Longleftrightarrow \hI\cap [j_{h_{a-1}}+1,j_{h_a}]\times[j_{h_b},j_{h_{b+1}}-1]\neq \emptyset.\\
			&\Longleftrightarrow [a,b]\in \phi^n_k(\hI,(j_{h_1},\dots,j_{h_k})).
		\end{align*}
		b):~The vector $\vj^{\pi}_k$ is the enumeration of the set $\pi^{-1}([k])$ and $\vj^{\pi}_{k+1}$ is the enumeration of the set $\pi^{-1}([k+1])$. Now $i_k=(\pi_{|k+1})^{-1}(k+1)$ is that position in $\vj^{\pi}_{k+1}$, that stems from the preimage of '$k+1$' under $\pi$. So if one removes the $i_k$-th value in $\vj^{\pi}_{k+1}$, one obtains  precisely $\vj^{\pi}_k$. The first equality $\hI_k=\phi^n_k(\hI, \vj^{\pi}_k)$ now follows from (\ref{eq:samedef}) together with (1) by induction.
		
		c):~For the second equation we first argue that it is enough to prove it for $k=n-1$. Define $f^n_k(\hI,\pi):=\pi_{|k}^{-1}(\pi(\hI)_{|k})$ and $g^n_k(\hI,\pi):=\phi^n_k(\hI,\vj^{\pi}_k)$. For every $1\leq k\leq m\leq n$ one obtains 
		$$f^m_k(f^n_m(\hI,\pi),\pi_{|m})=f^n_k(\hI,\pi)~~\text{and}~~g^m_k(g^n_m(\hI,\pi),\pi_{|m})=g^n_k(\hI,\pi),$$
		where the first equality follows easily from the definition. The second equality follows from b). So it is enough to prove the second equality for the case $k=n-1$: First relabeling $\hI$ with $\pi$ and then restricting $\pi(\hI)$ to the set $[n-1]$ results in the deletion of $\pi^{-1}(n)=i_{n-1}$. Then relabeling $\pi(\hI)_{|n-1}$ with the inverse of $\pi_{|n-1}$ results in reordering the set $\{1,2,\dots,i_{n-1}-1,i_{n-1}+1,\dots,n\}$ in its usual order. So the result is precisely $\phi^n_{n-1}(\hI,i_{n-1})$. 
	\end{proof}
	
	In particular, if $(I,\eta)=(I_n,\eta_n)_{n\in\bN}$ is an erased-interval process and $S=(S_n)_{n\in\bN}$ is the permutation process associated to $\eta$, then 
	\begin{equation}\label{eq:samplingas}
		I_k=\phi^n_k(I_n,\vec{j}^{S_n}_k)~~\text{almost surely for all}~1\leq k\leq n.
	\end{equation}
	In the next lemma we will state and prove some technical features involving the maps $\phi^n_k$ and $\phi^{\infty}_k$. In particular, it shows that one can interpret the maps $\phi^{\infty}_k$ to be the \emph{extensions of $\phi^n_k$ as $n\rightarrow\infty$}. It also explains why we have chosen our particular method of scaling finite interval systems and vectors from $[k:n]$.
	
	\begin{lemma}\label{lemma:technical}
		Let $n\in\bN$ and $1\leq k\leq n$. 
		\begin{enumerate}
			\item[i)] For all $\vec{j}\in[k:n]$ and $\hI\in\bIS(n)$ 
			\begin{equation*}
			\phi^n_k\left(\hI,\vj\right)=\phi^{\infty}_k\left(n^{-1}\hI,n^{-1}\vj\right).
			\end{equation*}
			\item[ii)] For all $K\in\bISi$ and $(u_1,\dots,u_n)\in[0,1]^n_{\neq}$ one has with $\pi:=\pi(u_1,\dots,u_n)$ 
			\begin{equation*}
			\phi^{\infty}_k\left(K,u_{1:k},\dots,u_{k:k}\right)=\phi^n_k\left(\phi^{\infty}_n(K,u_{1:n},\dots,u_{n:n}),\vj^{\pi}_k\right).
			\end{equation*}
			\item[iii)] The map $\phi^{\infty}_k$ is measurable with respect to the Borel $\sigma$-field on $\bISi\times[0,1]^k_<$.
		\end{enumerate}
	\end{lemma}
	
	\begin{proof}
		i):~For any $1\leq a<b\leq k$ it holds that 
		\begin{align*}
			[a,b]\in \phi^n_k\left(\hI,\vj\right)&\Longleftrightarrow \hI\cap [j_{a-1}+1,j_a]\times[j_{b},j_{b+1}-1]\neq \emptyset\\
												 &\Longleftrightarrow \text{there exists some~}[A,B]\in\hI~\text{such that}\\
												 &~~~~~~~~~~~~~~~~~~~~~~~~~j_{a-1}+1\leq A\leq j_a<j_b\leq B\leq j_{b+1}-1\\
												 &\Longleftrightarrow \text{there exists some~}[A,B]\in\hI~\text{such that}\\
												 &~~~~~~~~~~~~~~~\frac{2j_{a-1}-1}{2n}< \frac{A-1}{n}< \frac{2j_{a}-1}{2n}<\frac{2j_{b}-1}{2n}< \frac{B}{n}< \frac{2j_{b+1}-1}{2n}\\
												 &\Longleftrightarrow [a,b]\in \phi^{\infty}_k(n^{-1}\hI,n^{-1}\vj).
		\end{align*} 
		ii):~Let $\hI:=\phi^{\infty}_n(K,u_{1:n},\dots,u_{n:n})$. One obtains
		\begin{equation}\label{eq:orderedsamplingandrenaming}
			\pi(\hI)=\Big\{\{i\in[n]:x<u_i<y\}:(x,y)\in K\Big\}\cup\{\{j\}:j\in[n]\}\cup\{\emptyset\}.
		\end{equation}
		Thus restricting $\pi(\hI)$ to the set $\{1,\dots,k\}$ yields 
		\begin{equation}\label{eq:orderedsamplingandrenaming}
			\pi(\hI)_{|k}=\Big\{\{i\in[k]:x<u_i<y\}:(x,y)\in K\Big\}\cup\{\{j\}:j\in[k]\}\cup\{\emptyset\}.
		\end{equation}
		Hence the claimed equality follows from Lemma \ref{lemma:algorithmic}, b) and c).
		
		iii):~Fix some $1\leq a<b\leq k$. One needs to show that the set 
		$$A:=\{(K,u_1,\dots,u_k)\in\bISi\times[0,1]^k_<:K\cap (u_{a-1},u_a)\times(u_b,u_{b+1})\neq \emptyset\}$$
		is a Borel subset of $\bISi\times[0,1]^k_<$. For convenience we consider only the case $2\leq a<b\leq k-1$. Define 
		$$B:=\{(K,u_1,u_2,u_3,u_4)\in \bISi\times[0,1]^4_<:K\cap [u_1,u_2]\times[u_3,u_4]\neq \emptyset\}.$$
		This is a closed subset of $\bISi\times[0,1]^4_<$ and hence Borel. Let
		$$C:=\{(K,u_1,u_2,u_3,u_4)\in \bISi\times[0,1]^4_<:K\cap (u_1,u_2)\times(u_3,u_4)\neq \emptyset\}.$$
		One has that 
		$$(K,u_1,\dots,u_k)\in A~\Longleftrightarrow~(K,u_{a-1},u_a,u_{b},u_{b+1})\in C$$
		and so it is enough to argue that $C$ is Borel. For $n\in\bN$ define the set $C_n\subseteq \bISi\times[0,1]^4_<$ by 
		$$(K,u_1,u_2,u_3,u_4)\in C_n:\Leftrightarrow (K,u_1+n^{-1},u_2-n^{-1},u_3+n^{-1},u_4-n^{-1})\in B.$$
		Now every $C_n$ is closed and hence $B=\bigcup_{n\geq 1}C_n$ is Borel. 
	\end{proof}
	
	The next lemma shows that the Hausdorff distance between $K$ and $k^{-1}\phi^{\infty}_k(K,u_1,\dots,u_k)$ can be bounded uniformly in $K$ just depending on $(u_1,\dots,u_k)\in [0,1]^k_<$ in a non-trivial way:
	
	\begin{lemma}\label{bound}
		Let $k\in\bN$ and $u=(u_1,\dots,u_k)\in[0,1]^k_<$. Let $F_u$ be the empirical distribution function associated to $u$, so $F_u(x):=k^{-1}\sum\nolimits_{j=1}^k1(u_j\leq x)$ for $x\in[0,1]$ and let $F^{-}_u$ be the left-continuous version of $F_u$, so $F^{-}_u(x)=k^{-1}\sum\nolimits_{j=1}^k1(u_j<x)$. Let $K\in\bISi$. Then the sampled and normalized interval system $k^{-1}\phi^{\infty}_k(K,u_1,\dots,u_k)$ has the concrete representation
		$$k^{-1}\phi^{\infty}_k(K,u_1,\dots,u_k)=\Big\{\big(F_u(x),F^{-}_u(y)\big):(x,y)\in K, F_u(x)\leq F_u^{-}(y)\Big\}\cup \diago{}$$
		and the Hausdorff distance to $K$ can be bounded just in terms of $u$:
		$$\dha\Big(k^{-1}\phi^{\infty}_k(K,u_1,\dots,u_k),K\Big)\leq \sup\nolimits_{x\in[0,1]}|F_u(x)-x|+\sup\nolimits_{y\in[0,1]}|F^{-}_u(y)-y|.$$
	\end{lemma}
	\begin{proof}
		Since $[a,b]\in\phi^{\infty}_k(K,u_1,\dots,u_k)$ iff $K\cap (u_{a-1},u_a)\times(u_b,u_{b+1})\neq \emptyset$, every interval $[a,b]\in \phi^{\infty}_k(K,u_1,\dots,u_k)$ is of the form 
		$$\Big[1+\sum\nolimits_{j=1}^k1(u_j\leq x),\sum\nolimits_{j=1}^k1(u_j<y)\Big]~~\text{for some}~(x,y)\in K.$$
		Hence the scaling by $k$ yields the stated representation of $k^{-1}\phi^{\infty}_k(K,u_1,\dots,u_k)$. The distance bound is an easy consequence: For every $(x',y')\in K$ one has that 
		$$\min\limits_{(x'',y'')\in K}|x'-F_u(x'')|+|y'-F^{-}_u(y'')|\leq \sup\limits_{x\in[0,1]}|F_u(x)-x|+\sup\limits_{y\in[0,1]}|F^{-}_u(y)-y|,$$
		hence 
		$$d((x',y'),k^{-1}\phi^{\infty}_k(K,u_1,\dots,u_k))\leq \sup\limits_{x\in[0,1]}|F_u(x)-x|+\sup\limits_{y\in[0,1]}|F^{-}_u(y)-y|$$
		and thus 
		$$\max\limits_{(x',y')\in K}d((x',y'),k^{-1}\phi^{\infty}_k(K,u_1,\dots,u_k))\leq \sup\limits_{x\in[0,1]}|F_u(x)-x|+\sup\limits_{y\in[0,1]}|F^{-}_u(y)-y|.$$
		In the same way one can argue that 
		$$\max\limits_{(x',y')\in K}d((F_u(x'),F_u^{-}(y')),K)\leq \sup\limits_{x\in[0,1]}|F_u(x)-x|+\sup\limits_{y\in[0,1]}|F^{-}_u(y)-y|.$$
		This yields the distance bound. 
	\end{proof}
	
	We can now easily deduce that $n^{-1}I_n$ converges almost surely for every erased-interval process $(I_n,\eta_n)_{n\in\bN}$. Let $S$ be the permutation process corresponding to an eraser process $\eta$. For $1\leq k\leq n$ define 
	
	\begin{equation}\label{eq:defscaled}
		Y^n_k=(Y^n_{k,1},\dots,Y^n_{k,k}):=n^{-1}\vj^{S_n}_k,
	\end{equation}
	where $\vj^{S_n}_k\in[k:n]$ is the enumeration of the random $k$-set $S_n^{-1}([k])\subseteq[n]$. The scaled vector $Y^n_k$ takes values in $[0,1]^k_<$. Now let $(U_i)_{i\in\bN}$ be the $U$-process corresponding to $\eta$. One easily obtains that for every $k\in\bN$ 
	\begin{equation}\label{eq:almostsureconvergenceofsamplingsequence}
	Y^n_k=(Y^n_{k,1},\dots,Y^n_{k,k})\longrightarrow (U_{1:k},\dots,U_{k:k})~~\text{almost surely as $n\rightarrow\infty$}.
	\end{equation}
	
	Now we can prove the \emph{strong law of large numbers} for erased-interval processes:
	
	\begin{lemma}\label{lemma:asconvergence}
		Let $(I_n,\eta_n)_{n\in\bN}$ be an erased-interval process. Then $n^{-1}I_n$ converges almost surely in the space $(\bISi,\dha)$ towards some random variable $I_{\infty}$ as $n\rightarrow\infty$. 
	\end{lemma}
	\begin{proof}
		We will prove that $(n^{-1}I_n)_{n\in\bN}$ is a Chauchy sequence almost surely, which is sufficient since $(\bISi,\dha)$ is complete. Let $n\leq m$ and $Y^m_n=n^{-1}\vj^{S_m}_n$ be defined like in (\ref{eq:defscaled}), where $(S_n)_{n\in\bN}$ is the permutation process corresponding to $\eta$. By (\ref{eq:samplingas}) and Lemma \ref{lemma:technical} i) one obtains
		$$n^{-1}I_n=n^{-1}\phi^{\infty}_n(m^{-1}I_m,Y^m_n)~\text{almost surely.}$$
		Let $F^{-}_{Y^m_n}$ and $F_{Y^m_n}$ be the functions associated to $Y^n_m$ like in Lemma \ref{bound} which then yields
		$$\dha(n^{-1}I_n,m^{-1}I_m)\leq \sup\nolimits_{x\in[0,1]}|F_{Y^m_n}(x)-x|+\sup\nolimits_{y\in[0,1]}|F^{-}_{Y^m_n}(y)-y|.$$
		Now for every fixed $n$ the vector $Y^m_n$ converges almost surely towards $(U_{1:n},\dots,U_{n:n})$ as $m\rightarrow\infty$, where $U=(U_i)_{i\in\bN}$ is the $U$-process associated to $\eta$, see  (\ref{eq:almostsureconvergenceofsamplingsequence}). Since the uniform distribution on $[0,1]$ is diffuse, the Glivenko-Cantelli theorem yields that both $\sup\nolimits_{x\in[0,1]}|F_{Y^m_n}(x)-x|$ and $\sup\nolimits_{y\in[0,1]}|F^{-}_{Y^m_n}(y)-y|$ converge almost surely towards zero as $n,m\rightarrow\infty$. 
	\end{proof}
	
	Next we present a lemma which we will state in a slightly more general form and prove in Section \ref{sec:CUP}. The proof of the theorem presented there is based on topological features of isolated points in the \emph{Sorgenfrey plane}.
	
	If $A,B$ are events, we will say that $A$ implies $B$ almost surely iff $A\cap B^C$ has probability zero. We will denote this by $A\ias B$. Consequently we will say that $A$ and $B$ are equivalent almost surely iff the symmetric difference $A\Delta B=A\cap B^C\cup A^C\cap B$ has probability zero and we will denote this by $A\eas B$, so $A\eas B~\Longleftrightarrow~(A\ias B~~\wedge~~ B\ias A)$.
	
	\begin{lemma}\label{lemma:ractanglecut}
		Let $U=(U_i)_{i\in\bN}$ be a $U$-process, $n\geq 2$ and $0\leq j_1<j_2<j_3<j_3<j_4\leq n+1$. Let $Y_i:=U_{j_i:n}$ for $i\in\{1,2,3,4\}$. Consider the random rectangle
		\begin{align*}
		R&:=~\left[Y_{1},Y_{2}\right]\times\left[Y_{3},Y_{4}\right]=\Big\{(x,y)\in\recto{}:Y_1\leq x\leq Y_2~\text{and}~Y_3\leq y\leq Y_4\Big\}\\
		\intertext{and the interior of of that random rectangle}
		\inte(R)&:=\left(Y_{1},Y_{2}\right)\times\left(Y_{3},Y_{4}\right)=\Big\{(x,y)\in\recto{}:Y_1< x< Y_2~\text{and}~Y_3< y< Y_4\Big\}.
		\end{align*}
		Let $I_{\infty}$ be a random variable with values in $(\bISi,\dha)$ and independent of $U$. Then the sets $\{I_{\infty}\cap R\neq \emptyset\}$ and  $\{I_{\infty}\cap\inte(R)\neq\emptyset\}$ are events such that
		$$I_{\infty}\cap R\neq\emptyset~\overset{a.s.}{\Longrightarrow}~I_{\infty}\cap \inte(R)\neq \emptyset.$$
		In words: The random rectangle $R$ almost surely intersects $I_{\infty}$ in its interior or not at all. 
	\end{lemma}
	\begin{proof}
		See Section \ref{sec:CUP}.
	\end{proof}
	
	The next Lemma states that one can extend the almost sure equality $I_k=\phi^n_k(I_n,\vj^{S_n}_k)$ that holds almost surely for every $1\leq k\leq n$ to the case $n=\infty$. Hence in some sense, the maps $\phi^{\infty}_k$ are not just the algorithmic extension in the sense of Lemma \ref{lemma:technical}, but also the \emph{continuous extension}, continuous with respect to the randomized dynamics given by $\eta$. 
	
	\begin{lemma}\label{lemma:extensiontothelimit}
		Let $(I,\eta)$ be an erased-interval process and $I_{\infty}$ be the a.s. limit of $n^{-1}I_n$ according to Lemma \ref{lemma:asconvergence}. Let $(S_n)_{n\in\bN}$ and $(U_i)_{i\in\bN}$ be the permutation process and the $U$-process corresponding to $\eta$. For $n\geq k$ let $Y^n_k=n^{-1}\vj^{S_n}_k$ be like in (\ref{eq:almostsureconvergenceofsamplingsequence}). Then almost surely for every $k\in\bN$
		\begin{equation*}
		I_k=\lim\limits_{n\rightarrow\infty}\phi^{\infty}_k\left(n^{-1}I_n,Y^n_k\right)=\phi^{\infty}_k\Big(\lim\limits_{n\rightarrow\infty}\left(n^{-1}I_n,Y^n_k\right)\Big)=\phi^{\infty}_k\Big(I_{\infty},U_{1:k},\dots,U_{k:k}\Big).
		\end{equation*}
	\end{lemma}
	\begin{proof}
		Let $Y^n_k=(Y^n_{k,1},\dots,Y^n_{k,k})$. With $\widehat{I_k}:=\phi^{\infty}_k(I_{\infty},U_{1:k},\dots,U_{k:k})$ we need to prove that $I_k=\widehat{I_k}$ almost surely. Since both $I_k$ and $\widehat{I_k}$ are random sets, we show that $I_k\subseteq \widehat{I_k}$ and $\widehat{I_k}\subseteq I_k$ almost surely. Since singleton sets are part of every interval system by definition, we only need to consider intervals $[a,b]$ with $1\leq a<b\leq k$. 
		
		'$\subseteq$':~~We will show that $I_k\subseteq \widehat{I_k}$ almost surely. By Lemma \ref{lemma:technical} i) it holds that $I_k=\phi^{\infty}_k(n^{-1}I_n,Y^n_k)$ almost surely for every $n\geq k$ and hence by (\ref{eq:samplingREP}) one obtains
		$$[a,b]\in I_k~~\eas~~n^{-1}I_n~\cap~(Y^n_{k,a-1},Y^n_{k,a})\times(Y^n_{k,b},Y^n_{k,b+1})~\neq\emptyset~\text{for all}~~n\geq k.$$
		This yields
		$$[a,b]\in I_k~~\ias~~n^{-1}I_n~\cap~[Y^n_{k,a-1},Y^n_{k,a}]\times[Y^n_{k,b},Y^n_{k,b+1}]~\neq\emptyset~~\text{for all}~~n\geq k.$$
		Now $n^{-1}I_n\rightarrow I_{\infty}$ and $[Y^n_{k,a-1},Y^n_{k,a}]\times[Y^n_{k,b},Y^n_{k,b+1}]\rightarrow[U_{a-1:k},U_{a:k}]\times[U_{b-1:k},U_{b:k}]$ almost surely, where both convergences take place in the space $(\cK(\recto{}),\dha)$. One can easily check that if $(K_n)_{n\in\bN}$ and $(G_n)_{n\in\bN}$ are two sequences of compact sets converging towards $K$ and $G$ and if $K_n\cap G_n\neq \emptyset$ for all $n$, then also $K\cap G\neq \emptyset$. This yields
		$$[a,b]\in I_k~~\ias~~I_{\infty}~\cap~[U_{a-1:k},U_{a:k}]\times[U_{b:k},U_{b+1:k}]~\neq\emptyset.$$
		Now we apply Lemma \ref{lemma:ractanglecut} and obtain
		$$[a,b]\in I_k~~\ias~~I_{\infty}~\cap~(U_{a-1:k},U_{a:k})\times(U_{b:k},U_{b+1:k})~\neq\emptyset,$$
		so $[a,b]\in I_k\ias [a,b]\in\phi^{\infty}_k(I_{\infty},U_{1:k},\dots,U_{k:k})=\widehat{I_k}$ and thus $I_k\subseteq \widehat{I_k}$ almost surely.
		
		'$\supseteq$':~~We will show that $\widehat{I_k}\subseteq I_k$ almost surely. If $[a,b]\in \widehat{I_k}$ there is by definition some point $z=(x,y)\in I_{\infty}$ such that $U_{a-1:k}<x<U_{a:k}\leq U_{b:k}<y<U_{b+1:k}$, due to (\ref{eq:samplingREP}). Now since $I_{\infty}$ is the almost sure limit of $n^{-1}I_n$ the point $z$ is the limit of some sequence $z_n=(x_n,y_n)\in n^{-1}I_n$. Since $Y^n_k\rightarrow (U_{1:k},\dots,U_{k:k})$ almost surely one obtains that almost surely $Y^n_{a-1,k}<x_n<Y^n_{a,k}\leq Y^n_{b,k}<y_n<Y^n_{b+1,k}$ holds for all but finitely many $n\geq k$. In particular,
		$$[a,b]\in \widehat{I_k}~\ias~n^{-1}I_n\cap (Y^n_{a-1:k},Y^n_{a,k})\times(Y^n_{b:k},Y^n_{b+1,k})\neq \emptyset~\text{for some $n\geq k$}.$$
		And so again by (\ref{eq:samplingREP})
		$$[a,b]\in \widehat{I_k}~\ias~[a,b]\in\phi^{\infty}_k(n^{-1}I_n,Y^n_k)~\text{for some $n\geq k$}.$$
		But because $I_k=\phi^{\infty}_k(n^{-1}I_n,Y^n_k)$ almost surely for all $n\geq k$ by Lemma \ref{lemma:technical} i), one obtains 
		$$[a,b]\in\widehat{I_k}~\ias~[a,b]\in I_k$$
		and so $\widehat{I_k}\subseteq I_k$ almost surely. 
	\end{proof}

	The next lemma is used to obtain the topological description of the space of ergodic erased-interval processes. 
	
	\begin{lemma}\label{lemma:continuity}
		Let $(U_i)_{i\in\bN}$ be an $U$-process and $(K_n)_{n\in\bN}\subset\bISi$ be a sequence in $\bISi$ that converges towards some $K\in\bISi$. Then, for every $k\in\bN$, the sequence
		$\phi^{\infty}_k(K_n,U_{1:k},\dots,U_{k:k})$ converges almost surely towards $\phi^{\infty}_k(K,U_{1:k},\dots,U_{k:k})$ as $n\rightarrow\infty$.
	\end{lemma}
	\begin{proof}
		Since the RVs under consideration now take values in a discrete space, convergence of a sequence means that the sequence stays finally constant. We fix some $k\in\bN$ and some $1\leq a<b\leq k$. 
		
		If $[a,b]\in \phi^{\infty}_k(K,U_{1:k},\dots,U_{k:k})$, then $K\cap (U_{a-1:k},U_{a:k})\cap (U_{b:k},U_{b+1:k})\neq \emptyset$ by (\ref{eq:samplingREP}). Since $K_n\rightarrow K$, there is a sequence $z_n=(x_n,y_n)\in K_n$ that converges towards some point $z=(x,y)\in K$ with $U_{a-1:k}<x<U_{a:k}\leq U_{b:k}<y<U_{b+1:k}$. Then for all but finitely many $n$, the same inequality holds for $(x_n,y_n)$ instead of $(x,y)$ and so $[a,b]\in \phi^{\infty}_k(K_n,U_{1:k},\dots,U_{k:k})$ for all but finitely many $n$. Since interval systems always have finitely many elements, we have established that $\phi^{\infty}_k(K,U_{1:k},\dots,U_{k:k})\subseteq \phi^{\infty}_k(K_n,U_{1:k},\dots,U_{k:k})$ almost surely for all but finitely many $n$.
		
		Now let $[a,b]$ be such that $[a,b]\in\Phi^{\infty}_k(K_n,U_{1:k},\dots,U_{k:k})$ for infinitely many $n$. So there is a subsequence $(n_m)$ of $\bN$ and points $z_{n_m}=(x_{n_m},y_{n_m})\in K_{n_m}$ with 
		$U_{a-1:k}<x_{n_m}<U_{a:k}\leq U_{b:k}<y_{n_m}<U_{b+1:k}$ for all $m$. Now since $\recto{}$ is compact, the sequence $(z_{n_m})$ has a further subsequence that converges towards some $z=(x,y)$. Since $K_n\rightarrow K$, one has that $z\in K$ and furthermore $U_{a-1:k}\leq x\leq U_{a:k}\leq U_{b:k}\leq y\leq U_{b+1:k}$, so $K\cap [U_{a-1:k},U_{a:k}]\times[U_{b:k},U_{b+1:k}]\neq \emptyset$. Now with Lemma \ref{lemma:ractanglecut} one finally obtains 
		$$[a,b]\in \phi^{\infty}_k(K_n,U_{1:k},\dots,U_{k:k})~\text{for infinitly many $n$}~\ias~[a,b]\in\phi^{\infty}_k(K,U_{1:k},\dots,U_{k:k}).$$
		This completes the proof.
	\end{proof}
	
	Now we have all the ingredients we need to prove our first main theorem.
	
	\begin{proof}[Proof of Theorem \ref{thm:mainEIP}]
		Let $U=(U_i)_{i\in\bN}$ be an $U$-process and let $K\in\bISi$. Let $\eta$ be the eraser process corresponding to $U$. Then by the measurability of $\phi^{\infty}_k$ for every $k$ (Lemma \ref{lemma:technical}, iii)), the object 
		$$(I,\eta):=\Big(\phi^{\infty}_n(K,U_{1:n},\dots,U_{n:n}),\eta_n\Big)_{n\in\bN}$$
		introduced in (\ref{def:sampledEIP}) is a stochastic process. Let $\law(K)$ be its law. We will now argue that $\law(K)\in\ex(\cMI)$, so that the above defined process is an ergodic EIP. The first defining property of an erased-interval process is obvious. The third property follows from Lemma \ref{lemma:technical}, ii). So we need to show that $\eta_n$ is independent of $\cF_{n+1}=\sigma(I_m,\eta_m:m\geq n+1)$ for every $n$. By definition, $I_m$ is measurable with respect to $\sigma(U_{1:m},\dots,U_{m:m})$ and the latter is included in $\sigma(U_{1:n+1},\dots,U_{n+1:n+1},U_{n+2},U_{n+3},\dots)$ for every $m\geq n+1$. One has the almost sure equality of $\sigma$-fields
		$$\sigma(U_{1:n+1},\dots,U_{n+1:n+1})\vee\sigma(U_{n+2},U_{n+3},\dots)~~~\text{and}~~~\sigma(\eta_{n+1},\eta_{n+2},\dots).$$
		Since $\eta$ consists of independent RVs, $\eta_n$ is thus independent of $\cF_{n+1}$ for every $n\in\bN$, so $(I,\eta)$ really is an erased-interval process. 
		
		Now we will show that it is ergodic, so that $\cF_{\infty}=\bigcap\nolimits_{n\in\bN}\cF_n$ is trivial almost surely. By elementary arguments one can show that $\cF_{\infty}=\bigcap\nolimits_{n\in\bN}\sigma(I_m,\eta_m:m\geq n)$ is a.s. equal to $\bigcap\nolimits_{n\in\bN}\sigma(I_m:m\geq n)$. Since $I_m$ is measurable with respect to $\sigma(U_{1:m},\dots,U_{m:m})$ by construction, the latter $\sigma$-field is included in the exchangeable $\sigma$-field of $U$ and thus is trivial by Hewitt-Savage zero-one law. So we have proved that $\law(K)\in\ex(\cMI)$ for every $K\in\bISi$.
		
		Now let $(I,\eta)=(I_n,\eta_n)_{n\in\bN}$ be an arbitrary erased-interval process. By Lemma \ref{lemma:asconvergence} $n^{-1}I_n$ converges almost surely towards some $\bISi$-valued RV $I_{\infty}$ as $n\rightarrow\infty$. Let $U$ be the $U$-process corresponding to $\eta$ and $S$ the corresponding permutation process. For $n\in\bN$ define $\vec{S_n}:=n^{-1}(S_n^{-1}(1),S_n^{-1}(2),\dots,S_n^{-1}(n),0,0,\dots)\in[0,1]^{\bN}$, so $\vec{S_n}$ is considered to be a $[0,1]^{\bN}$-valued RV.  $\vec{S_n}$ converges almost surely towards $(U_1,U_2,\dots)$. Since $\vec{S_n}$ and $I_n$ are independent for every $n$, so are the a.s. limits $U$ and $I_{\infty}$.
		
		By Lemma \ref{lemma:extensiontothelimit} it holds that $I_k=\phi^{\infty}(I_{\infty},U_{1:k},\dots,U_{k:k})~~\text{almost surely for every $k\in\bN$}$. If $(I,\eta)$ is ergodic, then $I_{\infty}$ is almost surely constant. This yields that the map $\bISi\rightarrow\ex(\cMI), K\mapsto \law(K)$ is surjective. The map is also injective: For this it suffice to show that 
		$$\dha(k^{-1}\phi^{\infty}_k(K,U_{1:k},\dots,U_{k:k}),K)\rightarrow 0$$
		almost surely for $k\rightarrow\infty$. Since then, as limits are unique, $\law(K)$ and $\law(K')$ are clearly different for different $K,K'\in\bISi$. That the above Hausdorff distance tends to zero now follows easily with the general bound obtained in Lemma \ref{bound} and the Glivenko-Cantelli theorem. So we have proven that $K\mapsto \law(K)$ is bijective and Lemma \ref{lemma:extensiontothelimit} yields that every erased-interval process posses the described a.s. representation. 
		
		The last statement in Theorem \ref{thm:mainEIP} concerning the conditional distributions is immediate from the fact that the random objects $I_{\infty}$ and $U$ that occur in the a.s. representation are independent. 
		
		It remains to show that the map $K\mapsto \law(K)$ is a homeomorphism. Since it is bijective, $\bISi$ is compact and $\ex(\cMI)$ is Hausdorff, we only need to show that it is continuous. So we need to show that if $K_n\rightarrow K$ in $(\bISi,\dha)$ then $\law(K_n)\rightarrow \law(K)$. Fix some $U$-process and consider $(I^n,\eta)=(\phi^{\infty}_k(K_n,U_{1:k},\dots,U_{k:k}),\eta_k)_{k\in\bN}$, where $\eta$ is the eraser process corresponding to $U$. Now for every $n$ one has $\cL(I^n,\eta)=\law(K_n)$. By Lemma \ref{lemma:continuity} one has $(I^n,\eta)\rightarrow (I,\eta)$ almost surely as $n\rightarrow\infty$, where $(I,\eta)$ is constructed by sampling from $K$ via $U$. Now almost sure convergence implies convergence in law. 
	\end{proof}

	\begin{proof}[Proof of Theorem \ref{thm:mainMAP}]
		We first argue that $\cL(I,\eta)\mapsto \cL(H)$ is a surjective, affine and continuous map from $\cMI$ to $\cME$. 
		 
		So let $(I_n,\eta_n)_{n\in\bN}$ be an erased-interval process and let $S=(S_n)_{n\in\bN}$ be the permutation process corresponding to $\eta$. Let $H_n:=S_n(I_n)$. We first show that $H=(H_n)_{n\in\bN}$ is an exchangeable interval hypergraph on $\bN$. The exchangeability follows from the fact that $S_n$ is a uniform permutation independent of $I_n$. Now for every $\pi\in\bS_n$ one has $\pi(H_n)=\pi(S_n(I_n))=\pi\circ S_n(I_n)$. Now $\pi\circ S_n$ is again a uniform permutation independent of $I_n$, so $\pi(H_n)\sim H_n$ for every $n\in\bN$. Now we have that $I_n=\phi^{n+1}_n(I_{n+1},S_{n+1}^{-1}([n]))$ and the latter term is by Lemma \ref{lemma:algorithmic} equal to $(S_{n+1})_{|n}^{-1}((S_{n+1}(I_{n+1}))_{|n})$. Now since $S_n=(S_{n+1})_{|n}$  applying $S_n(\cdot)$ on both sides of $I_n=\phi^{n+1}_n(I_{n+1},S_{n+1}^{-1}([n]))$ almost surely yields $H_n=(H_{n+1})_{|n}$. Hence $H$ is an exchangeable interval hypergraph on $\bN$. The map $\cL(I,\eta)\mapsto \cL(H)$ is clearly affine and continuous, so we need to argue that it is surjective.
		
		Take some arbitrary exchangeable interval hypergraph $H=(H_n)_{n\in\bN}$ and perform the following steps:
		\begin{enumerate}
			\item As explained in Section \ref{sec:INTRO}, given some interval hypergraph $\hH\in\bIH(n)$ there exists some $\hI\in\bIS(n)$ and a permutation $\pi\in\bS_n$ such that $\hH=\pi(\hI)$. For every $\hH$ fix some $\hI_{\hH}\in\bIS(n)$ and a permutation $\pi_{\hH}\in\bS_n$ with $\hH=\pi_{\hH}(\hI_{\hH})$.
			\item Consider the sequence $\hI_{H_1},\hI_{H_2},\dots$ of random interval systems, so $\hI_{H_n}$ is a $\bIS(n)$-valued RV for every $n$. If $S_n$ is a uniform random permutation of $[n]$ independent of $H_n$, then $S_n(\hI_{H_n})$ has the same law as $H_n$: One has that $\hI_{H_n}=\pi_{H_n}^{-1}(H_n)$ and so $S_n(\hI_{H_n})=S_n\circ\pi_{H_n}^{-1}(H_n)$. The random permutation $S_n\circ \pi_{H_n}^{-1}$ is again uniform and independent of $H_n$ and the claim follows by exchangeability of $H_n$.
			\item Let $(S_n)_{n\in\bN}$ be a permutation process independent of $H$. For all $k,n\in\bN$ define
			$$I^n_k:=\begin{cases}
			\phi^n_k(\hI_{H_n},\vj^{S_n}_k),&\text{if $1\leq k\leq n$}\\
			\text{arbitrary element of $\bIS(k)$},&\text{else.}
			\end{cases}$$
			Now we have defined, for each $n$, a stochastic process $(I^n,\eta)=(I^n_k,\eta_k)_{k\in\bN}$ such that $I^n_k\in\bIS(k)$ for every $k$, where $\eta$ is the eraser process corresponding to $S$. The law of each process $(I^n,\eta)$ is a member of the compact metrizable space $\cM_1(\prod\nolimits_{k\in\bN}\bIS(k)\times[k+1])$. Denote the law of the $n$-th process by $L_n$. 
			\item The sequence $(L_n)_{n\in\bN}$ has a convergent subsequence $(L_{n_k})_{k\in\bN}$, let $L$ be its limit and $(I,\eta)=(I_k,\eta_k)_{k\in\bN}$ be a stochastic process with law $L$. 
			\item We claim that $(I,\eta)$ is an erased-interval process and that its law, namely $L$, serves as the desired preimage for $\cL(H)$ with respect to the map under consideration: 
			\begin{enumerate}
				\item $L\in\cMI$:~~For every fixed $n\in\bN$ the process $(\phi^n_k(\hI_{H_n},\vj^{S_n}_k),\eta_k)_{1\leq k\leq n}$ is a finite Markov chain with co-transition probabilities $\theta$ introduced in (\ref{eq:theta}). Now for every subsequence $n_k$ tending to infinity, the first $n_k$-component part of the law of $L_{n_k}$ is such a Markov chain with co-transitions given by $\theta$. Elementary arguments show that the limit law $L$ is thus in total the law of a Markov chain with co-transitions given by $\theta$, thus $L\in\cMI$. 
				\item By the algorithmic expression of $\phi^n_k$ presented in Lemma \ref{lemma:algorithmic} for every $1\leq k\leq n$ one obtains $S_k(I^n_k)=(S_n(\hI_{H_n}))_{|k}$. Now in (2) it was explained that $S_n(\hI_{H_n})$ has the same law as $H_n$ and so $S_k(I^n_k)$ has the same law as $H_k$. This proves that $\cMI\rightarrow\cME, \law(I,\eta)\mapsto \law(H)$ is surjective.
			\end{enumerate}
		\end{enumerate}
		The concrete representation of $H_n=S_n(I_n)$ follows directly from the definitions, it holds that $I_n=\phi^{\infty}_n(I_{\infty},U_{1:n},\dots,U_{n:n})$ almost surely and by that  
		\begin{align*}
		S_n(I_n)=\Bigg\{\Big\{j\in[n]:x<U_j<y\Big\}:(x,y)\in I_{\infty}\Bigg\}\cup\Big\{\{j\}:j\in[n]\Big\}\cup\{\emptyset\}
		\end{align*}
		almost surely. This completes the proof.
	\end{proof}
	
	\section{intersections of random sets}\label{sec:CUP}	
	
	In this section we will prove Lemma \ref{lemma:ractanglecut} which is used at two crucial points in the proof of Theorem \ref{thm:mainEIP}. For this we will establish the 
	following:
	
	\begin{proposition}\label{thm:AlmostSureInterxectionInInterior}
		Let $(Y_1,Y_2,Y_3,Y_4)$ be a $[0,1]^4_<$-valued random vector such that for all $i\neq j$ the conditional law of $Y_i$ given $Y_j$ is almost surely diffuse. Consider the random rectangle
		\begin{align*}
		R&:=~\left[Y_{1},Y_{2}\right]\times\left[Y_{3},Y_{4}\right]=\Big\{(x,y)\in\recto{}:Y_1\leq x\leq Y_2~\text{and}~Y_3\leq y\leq Y_4\Big\}\\
		\intertext{and the interior of of that random rectangle}
		\inte(R)&:=\left(Y_{1},Y_{2}\right)\times\left(Y_{3},Y_{4}\right)=\Big\{(x,y)\in\recto{}:Y_1< x< Y_2~\text{and}~Y_3< y< Y_4\Big\}.
		\end{align*}
		Let $K\subseteq\recto{}$ be any non-empty compact subset. Then the random rectangle $R$ almost surely intersects $K$ in its interior or not at all, more formally: For every $K\in\cK(\recto{})$ the set $\{K\cap R\neq \emptyset, K\cap\inte(R)=\emptyset\}$ is an event with 
		$\bP\big(K\cap R\neq \emptyset, K\cap\inte(R)=\emptyset\big)=0$.
	\end{proposition}

	See Fig. \ref{fig:RandomIntersection} for a visualization.

	\begin{figure}
		\centering
		\includegraphics[width=1.05\textwidth]{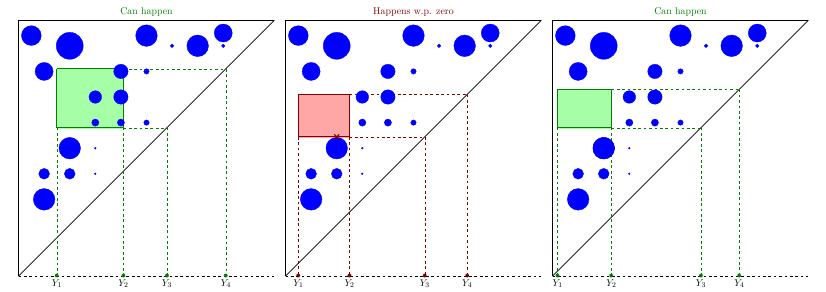}
		\caption{The compact set $K$ in blue. With probability one the left or the right case appears, provided the random rectangle has the above stated properties. The middle case, intersection only on the boundary, does not appear almost surely.}
		\label{fig:RandomIntersection}
	\end{figure}

	Our proof of this theorem relies on a topological feature of the \emph{Sorgenfrey plane}. The Sorgenfrey plane is a topological space $\cS^2$ on the set of points $\bR^2$ where we choose the set of rectangles of the form $[a,b)\times[c,d)$ with $a<b$ and $c<d$ as a basis for the topology. This topology really refines the usual topology on $\bR^2$, thus \emph{more subsets of $\bR^2$ are open in $\cS^2$}. In fact, the topological space $\cS^2$ is no longer 'nice': Although it is separable, it is not metrizable. Because there are more open sets, it could happen that in a given subset $A\subseteq \bR^2$ more points $x\in A$ are \emph{isolated} than in the euclidean case. A point $x\in A$ is called isolated, if there is an open set $U\subseteq \bR^2$ with $U\cap A=\{x\}$. Denote the set of isolated points of $A\subset \bR^2$ by $\iso(A)$. In the usual euclidean topology on $\bR^2$ sets of isolated points are at most countable. This feature is lost in $\cS^2$, in fact the set of isolated points may well be uncountable. For example in the uncountable set $\{(x,-x):x\in\bR\}\subseteq\bR^2$ \emph{every point} is isolated: for some point $(x,-x)$ just take the isolating and open neighborhood $[x,x+1)\times[-x,-x+1)$. However, we will show that the set of isolated points for any given subset in the Sorgenfrey plane is 'small enough for our purposes': The set $\{(x,-x):x\in\bR\}$, although uncountable, is just the graph of the strictly decreasing function $f(x)=-x$ and a similar feature holds for any given set of isolated points in the Sorgenfrey plane. We do not know if the following proposition is a new result, hence we prove it:
	
	\begin{proposition}\label{thm:isolatedpoints}
		For every subset $A\subseteq\bR^2$ of the Sorgenfrey plane the set of isolated points $\iso(A)$ can be covered by the union of countably many graphs of strictly decreasing functions $f_i:\bR\rightarrow\bR, i\in\bN$.
	\end{proposition}
	
	\begin{proof}
		Let $\varepsilon>0$ and define the subset $\iso^{\varepsilon}(A)$ of $A$ by 
		$$z\in \iso^{\varepsilon}(A)~:\Longleftrightarrow~z=(x,y)\in A~\text{and}~[x,x+\varepsilon)\times[y,y+\varepsilon)\cap A=\{z\}.$$
		Now $z\in A$ is isolated iff there is some $\varepsilon>0$ such that $z\in\iso^{\varepsilon}(A)$, in particular 
		$$\iso(A)=\bigcup\nolimits_{n\in\bN}\iso^{1/n}(A).$$
		Fix $\varepsilon>0$. For some point $z=(x,y)\in\bR^2$ define 
		$$U(z,\epsilon):=\Big[x-\frac{\varepsilon}{4},x+\frac{\varepsilon}{4}\Big]\times\Big[y-\frac{\varepsilon}{4},y+\frac{\varepsilon}{4}\Big],$$
		so $U(z,\epsilon)$ is a closed square that has $z$ as its midpoint and whose edges are of length $\varepsilon/2$. With this definitions one obtains that for every $z\in\bR^2$ and every two different points $(a_1,a_2),(b_1,b_2)\in\iso^{\varepsilon}(A)\cap U(z,\varepsilon)$: 
		$$a_1<b_1~\text{and}~a_2>b_2~~~~~\text{or}~~~~~b_1<a_1~\text{and}~b_2>a_2.$$
		So the set $\iso^{\varepsilon}(A)\cap U(z,\varepsilon)$ can be covered by the graph of some strictly decreasing function $f_{z,\varepsilon}:[x-\frac{\varepsilon}{4},x+\frac{\varepsilon}{4}]\rightarrow[y-\frac{\varepsilon}{4},y+\frac{\varepsilon}{4}]$.
		Now for every $\varepsilon>0$ one can choose countably many $z^{\varepsilon}_1,z^{\varepsilon}_2,\dots\in\bR^2$ such that $\bR^2=\bigcup\nolimits_{i\in\bN}U(z^{\varepsilon}_i,\varepsilon).$ This implies that $\iso^{\epsilon}(A)$ is covered by the union of the graphs of the functions $f_{z^{\varepsilon}_i,\varepsilon}, i\in\bN$. Consequently, the whole of $\iso(A)$ is contained in the union of all graphs of the functions $f_{z^{1/n}_i,1/n}, n,i\in\bN$.
	\end{proof}

	\begin{wrapfigure}{r}{0.15\textwidth}
		\centering
		\includegraphics[width=0.15\textwidth]{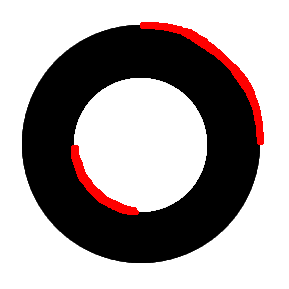}
		\caption{\label{fig:isolatedpoints}}
	\end{wrapfigure}

	We could introduce 'tilted' Sorgenfrey planes by choosing rectangles of the form $(a,b]\times (c,d]$ or $[a,b)\times (c,d]$ or $(a,b]\times[c,d)$ as a basis for the topology. Of course, the analogue statement of Proposition~\ref{thm:isolatedpoints} would be true for these as well, one would only need to interchange 'strictly decreasing' with 'strictly increasing' in the latter two cases. In Fig. \ref{fig:isolatedpoints} one can see a subset $K\subseteq\bR^2$ where the isolated points of $K$ are highlighted in red. Here two strictly decreasing functions are sufficient to cover the isolated points. The union of the isolated points w.r.t. to all four tilted Sorgenfrey planes would cover the whole (euclidean) boundary of this set $K$.

	\begin{proof}[Proof of Proposition \ref{thm:AlmostSureInterxectionInInterior}]
		Fix some $K\in\cK(\recto{})$. We first prove that $\{K\cap R\neq \emptyset, K\cap\inte(R)=\emptyset\}$ is an event. For this we introduce
		$$A:=\Big\{u=(u_1,u_2,u_3,u_4)\in[0,1]^4_<:K\cap [u_1,u_2]\times[u_3,u_4]\neq \emptyset\Big\}$$
		and 
		$$B:=\Big\{u=(u_1,u_2,u_3,u_4)\in[0,1]^4_<:K\subseteq ((u_1,u_2)\times(u_3,u_4))^C\Big\}.$$
		Then 
		$$\{K\cap R\neq \emptyset, K\cap\inte(R)=\emptyset\}=\{(Y_1,Y_2,Y_3,Y_4)\in A\cap B\}.$$
		Now both $A$ and $B$ are closed subsets of $[0,1]^4_<$: Suppose $u^n=(u^n_1,u^n_2,u^n_3,u^n_4)$ is a sequence in $A$ converging towards some $u=(u_1,u_2,u_3,u_4)\in[0,1]^4_<$. By definition of $A$ for every $n$ there is a point $y^n=(y^n_1,y^n_2)\in K$ such that $u^n_1\leq y^n_1\leq u^n_2$ and $u^n_3\leq y^n_2\leq u^n_4$. Since $K$ is compact there exists a converging subsequence $y_{n_k}$ with limit $y=(y_1,y_2)\in K$. Since $u^n\rightarrow u$ it holds that $u_1\leq y_1\leq u_2$ and $u_3\leq y_2\leq u_4$. So $y\in K\cap [u_1,u_2]\times[u_3,u_4]$ and thus $u\in A$. With the same basic considerations one can prove that $B$ is closed. Hence $\{K\cap R\neq \emptyset, K\cap\inte(R)=\emptyset\}$ is an event. 
		
		We will now detect the points in $K$ that can be hit by rectangles on the boundary but not in the interior. Let us introduce the subset $W$ of $K$ as the set of all points $(x,y)\in K$ that are 'isolated on the west', meaning there exists some open rectangle $r$ where the closure of the 'west' side of $r$ contains $(x,y)$ away from its corners and that is disjoint from $K$. Formally,
		$$W:=\Bigg\{(x,y)\in K:~\begin{aligned}&\text{there are}~(b,c,d)~\text{s.t.}~0\leq x<b<c<y<d\leq 1\\ &\text{and}~(x,b)\times(c,d)\cap K=\emptyset
		\end{aligned}\Bigg\}.$$
		In the same way we introduce the set of points of $K$ that are isolated on the east, north or to the south (denoted by $E,N$ and $S$). The points in $K$ that can be hit by rectangles on the interior of the four boundary sides but not in the interior of the rectangle are given by $E\cup W\cup N\cup S$.
		
		Rectangles could hit points in $K$ on the corners but not on the interior. We define the set of points that could be hit by the south-west corner of a rectangle but not in its interior to be $SW\subseteq K$, formally:
		$$SW:=\Bigg\{(x,y)\in K:~\begin{aligned}&\text{there are}~(b,d)~\text{s.t.}~0<x<b<y<d<1\\ &\text{and}~[x,b)\times[y,d)\cap K=\{(x,y)\}
		\end{aligned}\Bigg\}.$$
		In the same way we introduce the sets $SE,NW$ and $NE$. Now we can characterize the event under consideration: for that we introduce the projections $\pi_1:\recto{}\rightarrow [0,1]$ and $\pi_2:\recto{}\rightarrow[0,1]$ by $\pi_1((x,y))=x$ and $\pi_2((x,y))=y$. Now we claim:
		\\
		
		($\star$) For every $(u_1,u_2,u_3,u_4)\in[0,1]^4_<$ such that $[u_1,u_2]\times[u_3,u_4]\cap K\neq \emptyset$ and $(u_1,u_2)\times(u_3,u_4)\cap K=\emptyset$ at least one of the following eight statements is true:
		\begin{align*}
		&u_1\in\pi_1(W),u_2\in\pi_1(E),u_3\in\pi_2(S),u_4\in\pi_2(N),\\
		&(u_1,u_3)\in SW,(u_2,u_3)\in SE,(u_1,u_4)\in NW,(u_2,u_4)\in NE.
		\end{align*}
		
		The sets $\pi_1(W),\pi_1(E),\pi_2(S)$ and $\pi_2(N)$ are all at most countable infinite, which can be seen quite easily.
		
		Now observe that $SW\subseteq \iso(K)$ so by Proposition \ref{thm:isolatedpoints} there are countable many strictly monotone functions $f^{sw}_i:\bR\rightarrow\bR$ such that $SW\subseteq\iso(K)\subseteq \bigcup\nolimits_{i\in\bN}\graph(f^{sw}_i)$. The sets $SE,NW$ and $NE$ are contained in the isolated points of $K$ with respect to the above mentioned tilted Sorgenfrey planes, so in each case there are countable many strictly monotone functions $f^{se}_i,f^{nw}_i,f^{ne}_i,\in\bN$ whose graphs cover the corresponding sets. 
		Since any monotone function $f:\bR\rightarrow\bR$ is measurable, the graph $\graph(f)$ of any monotone function $f$ is a Borel subset of $\bR^2$. Thus by using $(\star)$ and the union bound for probabilities we arrive at the following upper bound for the probability we are interested in:
		\begin{align*}
		\bP\big(K\cap R&\neq \emptyset, K\cap\inte(R)=\emptyset\big)\leq \\
		&\bP\big(Y_1\in \pi_1(W)\big)+\bP\big(Y_2\in \pi_1(E)\big)+\bP\big(Y_3\in \pi_2(S)\big)+\bP\big(Y_4\in \pi_2(N)\big)\\
		&+\sum\nolimits_{i=1}^{\infty}\Big[\bP\big(Y_3=f^{sw}_i(Y_1)\big)+\bP\big(Y_3=f^{se}_i(Y_2)\big)\\
		&~~~~~~~~~~~~~~~~~~~~~~~~~~~~~~~~~~~+\bP\big(Y_4=f^{nw}_i(Y_1)\big)+\bP\big(Y_4=f^{ne}_i(Y_2)\big)\Big].
		\end{align*}
		Now our assumptions on the law of $(Y_1,Y_2,Y_3,Y_4)$ are needed to conclude that each of the probabilities occurring above is zero: Since the conditional law of $Y_i$ given $Y_j$ is almost surely diffuse, the unconditional law of each $Y_i$ is diffuse. Since the projection sets are countable, the first four probabilities are zero. Now we take a look at $\bP(Y_3=f^{sw}_i(Y_1))$:
		$$\bP(Y_3=f^{sw}_i(Y_1))=\int_{[0,1]}\bP(Y_3=f^{sw}_i(y)|Y_1=y)d\bP^{Y_1}(y)=0,$$
		since the conditional law of $Y_3$ given $Y_1=y$ is diffuse for $\bP^{Y_1}$-almost all $y\in[0,1]$. The same reasoning holds for every other remaining term.
	\end{proof}
	
	The lemma used in Section \ref{sec:proof} was a little different from Proposition \ref{thm:AlmostSureInterxectionInInterior}, but can now be easily deduced from it. 
	
	\begin{proof}[Proof of Lemma \ref{lemma:ractanglecut}]
		First assume that $n\geq 4$ and $1\leq j_1<j_2<j_3<j_4\leq n$. In this case the random vector $(Y_1,Y_2,Y_3,Y_4)$ satisfies the assumptions of Proposition \ref{thm:AlmostSureInterxectionInInterior}. Now the only difference is that the compact set under consideration may be random, so we need to make sure that we really deal with an event. Let
		$$A:=\Big\{(K,u_1,u_2,u_3,u_4)\in\cK(\recto{})\times[0,1]^4_<:K\cap [u_1,u_2]\times[u_3,u_4]\neq \emptyset\Big\}$$
		and 
		$$B:=\Big\{(K,u_1,u_2,u_3,u_4)\in\cK(\recto{})\times[0,1]^4_<:K\cap (u_1,u_2)\times(u_3,u_4)\neq \emptyset\Big\}.$$
		As in the proof of Theorem \ref{thm:AlmostSureInterxectionInInterior} one easily obtains that both $A$ and $B^C$ are closed subsets of $\cK(\recto{})\times[0,1]^4_<$, hence we really deal with events and the result follows easily from Theorem \ref{thm:AlmostSureInterxectionInInterior} and the assumed independence of $I_{\infty}$ and $U$, using Fubinis theorem.
		
		Now to the case $j_1=0$ and/or $j_4=n+1$. Assume that $j_1=0$ and $j_4\leq n$. Since we have defined $U_{0:n}=-1$ the random rectangle $R$ can not intersect $I_{\infty}$ at its left side or at one of its two left corners, they do not belong to $\recto{}$. So if an intersection on the boundary of the rectangle takes place, coordinates have to be involved that satisfy the assumptions of the almost sure diffuseness and we refer to the arguments presented in the proof of Proposition \ref{thm:AlmostSureInterxectionInInterior}. The same strategy succeeds in the cases $j_1\geq 1$ and $j_4=n+1$ or $j_1=0$ and $j_4=n+1$. In the latter case one only needs to argue with the RVs $U_{j_2}$ and $U_{j_3}$ which always fulfill the almost sure-diffuseness assumption. 
	\end{proof}
	
	\section{applications}\label{sec:connections}
	
	In this section we will connect our results for exchangeable interval hypergraphs and erased-interval processes to exchangeable hierarchies on $\bN$ in the sense of \cite{fohapi}, to the Martin boundary of R\'emy's tree growth chain in the sense of \cite{egw2} and to composition structures in the sense of \cite{gnedin}. At the end, we will present an outlook for future research.
	
	\subsection{Hierarchies and Schr{\"o}der trees}
	
	\begin{definition}
		A \emph{hierarchy} on $[n]$ is a subset $\hH\subseteq \cP([n])$ such that $\emptyset\in\hH,\{j\}\in\hH$ for every $j\in[n]$, $[n]\in\hH$ and such that for all $e,f\in\hH$ it holds that $e\cap f\in\{e,f,\emptyset\}$. Let $\bH(n)$ be the set of all hierarchies on $[n]$.	
	\end{definition}
	 Hierarchies on $[n]$ are equivalent to leaf-labeled unordered rooted trees in which every internal node has at least two descendants (see \cite{fohapi}). Every such tree can be embedded into the plane: for every internal node one chooses an ordering on the descendants. Now the nodes of that ordered tree get equipped with the canonical lexicographic ordering. This yields a linear order $l$ on $[n]$: $i$ is smaller then $j$ w.r.t. $l$ if and only if the leaf labeled with $i$ is smaller than the leaf labeled with $j$ with respect to the lexicographic ordering. Thereby the hierarchy $\hH$ becomes an interval hypergraph. So $\bH(n)\subseteq \bIH(n)$ for every $n$. Hierarchies are closed under restriction and relabeling, this is immediate from the definition. 
	 \begin{definition}
	 	An exchangeable hierarchy on $\bN$ is an exchangeable interval hypergraph $(H_n)_{n\in\bN}$ such that $H_n\in\bH(n)$ for every $n$. Let $\cMEH$ be the space of all possible laws of exchangeable hierarchies on $\bN$. 
	 \end{definition}
 	
 	In \cite{fohapi} the authors provided two de Finetti-type characterization theorems for exchangeable hierarchies on $\bN$: At first they worked out a description via \emph{sampling from real trees}. Given any exchangeable hierarchy on $\bN$ they constructed a real tree and a probability measure concentrated on the leafs of that tree and then that they proved that the law of the exchangeable hierarchy is the same as the law of the sequence of finite combinatorial subtrees obtained by sampling at iid position  according to the probability measure on that tree. We will not give further details here and refer the reader to \cite[Theorem 5]{fohapi}. From this result they obtained a second representation result, sampling from interval hierarchies on $[0,1)$: these are subsets $\cH\subseteq \cP([0,1))$ such that every $e\in \cH$ is an interval (w.r.t. to the usual linear order on $[0,1)$), such that $\{x\}\in \cH$ for all $x\in [0,1)$, $[0,1)\in\cH$ and such that $e,f\in \cH$ implies $e\cap F\in\{\emptyset,e,f\}$. Denote the space of all interval hierarchies on $[0,1)$ by $\bIHH$. This space was then equipped with a measurability structure: The authors considered the $\sigma$-field generated by restriction to finite sets, that is given some $\cH\in\bIHH$ and a finite subset $A\subset [0,1)$ they considered $\cH_{|A}:=\{A\cap e:e\in \cH\}$. Their representation result reads as follows, where we restate it in a slightly different but equivalent form in which tail $\sigma$-fields are replaced by exchangeable $\sigma$-fields:
 	
 	\begin{theorem*}[\cite{fohapi}, Theorem $4$]
 		Let $(U_i)_{i\in\bN}$ be an $U$-process and let $\cH\in\bIHH$. Let $\law^{\text{ih}}(\cH)$ be the distribution of  
 		$$\Big(\big\{\{i\in[n]:U_i\in e\}:e\in \cH\big\}\cup\{\emptyset\}\Big)_{n\in\bN}.$$
 		Then 
 		\begin{enumerate}
 			\item[a)] $\law^{\text{ih}}(\cH)\in\ex(\cMEH)$ for every $\cH\in\bIHH$ and the map $$\bIHH\rightarrow\ex(\cMEH), H\mapsto \law^{\text{ih}}(\cH)$$ is surjective.
 			\item[b)] For any exchangeable hierarchy $H=(H_n)_{n\in\bN}$ on $\bN$ there is a random $H$-measurable interval hierarchy $\cH$ on $[0,1)$ such that the conditional law of $H$ given the exchangeable $\sigma$-field of $H$ is almost surely equal to $\law^{\text{ih}}(\cH)$. 
 		\end{enumerate} 
 	\end{theorem*}
 	
 	The map in $a)$ is far from being injective. In fact, the cardinality of $\bIHH$ is strictly larger than the cardinality of $\ex(\cMEH)$. Hence it is not possible to introduce a metric on $\bIHH$ that would turn it into a complete separable metric space. We will offer an improvement of statement a) below that avoids this. For this we obverse that $\cMEH$ is a simplex and by definition, it is a subset of $\cME$. But it is not just included: $\cMEH$ is a \emph{closed face} in the simplex $\cME$, in particular  $\ex(\cMEH)\subseteq\ex(\cME)$. We will use this fact to deduce a representation result concerning exchangeable hierarchies from our representation result concerning $\cME$. We will perform this deduction by passing to erased-type objects at first: For some linear order $l$ on $[n]$ let $\bH(n,l)$ be the set of all hierarchies that are interval hypergraphs w.r.t. $l$. As it is the case with interval hypergraphs, for every $n$ and every $\hH\in\bH(n)$ there is some bijection $\pi$ such that $\pi(\hH)$ is an interval hypergraph w.r.t. the usual linear order $<$. 
	\begin{definition}
		A hierarchy $\hT$ on $[n]$ that is an interval hypergraph w.r.t. to the usual linear order $<$ is called a \emph{Schr{\"o}der tree}. Let $\bST(n):=\bH(n,<)$ be the set of Schr{\"o}der trees on $[n]$. 
	\end{definition}
	Schr{\"o}der trees on $[n]$ are usually introduced as rooted ordered trees with exactly $n$ leafs in which every internal node has at least two descendants (see \cite{ARS}). Our definition is equivalent: Given any rooted ordered tree with exactly $n$ leafs one can enumerate the leafs from $1,\dots,n$ in the lexicographic ordering. Now to every node of the tree we attach the set of numbers of those leaves that are descendants of that node. Every such set of nodes is an interval. We collect all these intervals into a set and include the empty set. The result is an element of $\bST(n)$, that determines the tree-structure in a unique way. By definition every Schr{\"o}der tree on $[n]$ is an interval system on $[n]$ as well, so $\bST(n)\subseteq \bIS(n)$. One can directly see whether an element $\hI\in\bIS(n)$ is a Schr{\"o}der tree: this is the case if and only if for every $[a_1,b_1],[a_2,b_2]\in \hI$ with $a_1<a_2\leq b_1$ it holds that $b_2\leq b_1$, i.e. iff intervals do not overlap. This reflects the property that, in any tree, different subtrees are either disjoint or included. Schr{\"o}der trees are stable under removing elements according to $\phi^{n+1}_n$. If $\hT\in\bST(n)$ and $\vj\in[k:n]$, then $\phi^n_k(\hT,\vj)\in\bIS(k)$ is the ordered subtree induced at the leaves $\vj$. The leaves are then renamed by $[k]$ in a strictly increasing manner; see Fig. \ref{fig:schroedertree} for a visualization of some Schr{\"o}der tree. 

	\begin{figure}
		\centering
		\includegraphics[width=0.95\textwidth]{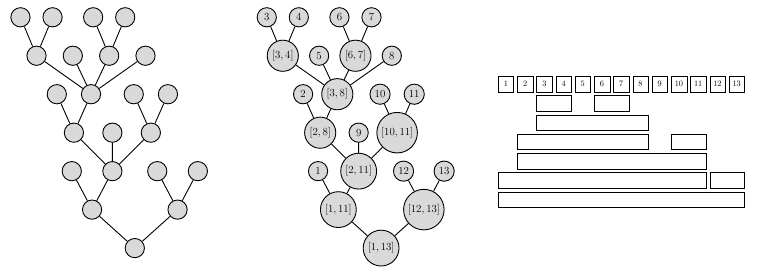}
		\caption{On the left a Schr{\"o}der tree with $13$ leafs. Next the canonical labeling of that tree obtained from the lexicographic order of the leafs. On the right the representation as an interval system.}
		\label{fig:schroedertree}
	\end{figure}
	
	\begin{definition}
		An \emph{erased-Schr{\"o}der tree process} is an erased-interval process $(T_n,\eta_n)_{n\in\bN}$ such that $T_n\in\bST(n)$ for every $n$. Let $\cMS$ be the space of all possible laws of erased-Schr{\"o}der tree processes. 
	\end{definition}
	As in the case above, $\cMS$ is not just a simplex and a subset of $\cMI$, but also a closed face in $\cMI$; in particular, $\ex(\cMS)\subseteq \ex(\cMI)$. Since we have identified $\ex(\cMI)$ with the space $(\bISi,\dha)$, we only have to find that subspace of $\bISi$ that yields Schr{\"o}der trees. Since $\bISi$ was the analogue of $\bIS(n)$ with $n\rightarrow\infty$, the analogue for $\bST(n)$ with $n\rightarrow\infty$ is straightforward to obtain: 
	
	\begin{definition}
		$K\in\bISi$ is called a \emph{Schr{\"o}der tree on $(0,1)$} iff $(0,1)\in K$ and for every $(x_1,y_1),(x_2,y_2)\in K$ with $x_1<x_2<y_1$ it holds that $y_2\leq y_1$. Denote by $\bSTi$ the set of all Schr{\"o}der trees on $(0,1)$.		
	\end{definition} 

	\begin{lemma}\label{lemma:bst}
		Schr{\"o}der trees form a substructure of interval systems that satisfy the following consistency properties:
		\begin{enumerate}
			\item[i)] $\bSTi$ is a closed subset of $\bISi$.
			\item[ii)]  If $K\in\bSTi, k\in\bN, (u_1,\dots,u_k)\in[0,1]^k_<$ then $\phi^{\infty}_k(K,u_1,\dots,u_k)\in\bST(k)$.
			\item[iii)]  For $n\in\bN$ and $\hT\in\bST(n)$ it holds that $n^{-1}\hT\in\bSTi$. 
		\end{enumerate}
	\end{lemma}
	\begin{proof}
		i):~Let $(K_n)_{n\in\bN}$ be a convergent sequence in $\bSTi$ with limit $K\in\bISi$. We need to show that $K$ is a Schr{\"o}der tree on $(0,1)$. It is obvious that $(0,1)\in K$. Let $z_1=(x_1,y_1),z_2=(x_2,y_2)\in K$ with $x_1<x_2<y_1$. Since $\dha(K_n,K)\rightarrow 0$ there are  $z^n_1=(x^n_1,y^n_1),z^n_2=(x^n_2,y^n_2)\in K_n$ such that $z^n_1\rightarrow z_1$ and $z^n_2\rightarrow z_2$ as $n\rightarrow\infty$. Since $x_1<x_2<y_1$ it holds that $x^n_1<x^n_2<y^n_1$ for all but finitely many $n$. Since all $K_n$ are Schr{\"o}der trees, it follows that $y^n_2\leq y^n_1$ for all but finitely many $n$. This implies $y_2\leq y_1$. Hence $K$ is a Schr{\"o}der tree. 
		
		ii):~Let $\hT:=\phi^{\infty}_k(K,u_1,\dots,u_k)$ and $[a_1,b_1],[a_2,b_2]\in \hT$ with $a_1<a_2\leq b_1$. Hence there are some $(x_1,y_1),(x_2,y_2)\in K$ such that 
		$$u_{a_1-1}<x_1<u_{a_1}<u_{b_1}<y_1<u_{b_1+1}~~\text{and}~~u_{a_2-1}<x_2<u_{a_2}<u_{b_2}<y_2<u_{b_2+1}.$$
		Since $a_1\leq a_2-1$ it holds that $u_{a_1}\leq u_{a_2-1}$ and so $x_1<x_2$. Further, since $a_2\leq b_1$ it holds that $u_{a_2}\leq u_{b_1}$ and so $x_2<y_1$. So $x_1<x_2<y_1$ and because $K$ is assumed to be a Schr{\"o}der tree it holds that $y_2\leq y_1$. Hence $u_{b_2}<u_{b_1+1}$ and so $b_2\leq b_1$. Furthermore it holds that $[n]\in\hT$, since $(0,1)\in K$. This shows that $\hT$ is a Schr{\"o}der tree. 
		
		iii):~Let $(x_1,y_1),(x_2,y_2)\in n^{-1}\hT$ with $x_1<x_2<y_1$. By definition of $n^{-1}\hT$ there are some $[a_1,b_1],[a_2,b_2]\in\hT$ with $(x_i,y_i)=((a_i-1)/n,b_i/n)$. One obtains $a_1<a_2\leq b_1$. Since $\hT$ is a Schr{\"o}der $b_2\leq b_1$ and hence $y_2\leq y_1$. Furthermore, because $[n]\in\hT$ also $(0,1)\in n^{-1}\hT$. Hence $n^{-1}\hT\in \bSTi$. 
	\end{proof}

	Lemma \ref{lemma:bst} and Theorem \ref{thm:mainEIP} directly yield a concrete description of erased-Schr{\"o}der tree processes and in that way a description of exchangeable hierarchies on $\bN$. The latter serves as an improvement of part a) of the theorem given in \cite{fohapi}. 
	
	\begin{corollary}\label{thm:mainEST}
		For every $K\in\bSTi$ one has $\law(K)\in\ex(\cMS)$ and the map $\bSTi\rightarrow\ex(\cMS),K\mapsto \law(K)$ is a homeomorphism. One has the following concrete representation: Let $(T,\eta)=(T_n,\eta_n)_{n\in\bN}$ be an erased-Schr{\"o}der tree process. Then $n^{-1}T_n$ converges almost surely as $n\rightarrow\infty$ towards some $\bSTi$-valued random variable $T_{\infty}$. Let $U=(U_i)_{i\in\bN}$ be the $U$-process corresponding to $\eta$. Then $T_{\infty}$ and $U$ are independent and one has the equality of processes
		\begin{equation*}
		(T_n,\eta_n)_{n\in\bN}~=~\Big(\phi^{\infty}_n(T_{\infty},U_{1:n},\dots,U_{n:n}),\eta_n\Big)_{n\in\bN}~~\text{almost surely.}
		\end{equation*}
		In particular, for every erased-Sch{\"o}der tree process $(T,\eta)$ the conditional law of $(T,\eta)$ given the terminal $\sigma$-field $\cF_{\infty}$ is $\law(T_{\infty})$ almost surely and $T_{\infty}$ generates $\cF_{\infty}$ almost surely.
	\end{corollary} 
	
	See Fig. \ref{fig:phiinftyB} for an illustration of the overall procedure.

	\begin{corollary}\label{thm:mainEIHier}
		Let $(U_i)_{i\in\bN}$ be an $U$-process and let $K\in\bSTi$. Let $\law^{\text{ih}}(K)$ be the distribution of  
		$$\Big(\big\{\{j\in[n]:x<U_j<y\}:(x,y)\in K\big\}\cup\big\{\{j\}:j\in[n]\big\}\cup\{\emptyset\}\Big)_{n\in\bN}$$
		Then the following holds
		\begin{enumerate}
			\item[a')] $\law^{\text{ih}}(K)\in\ex(\cMEH)$ for every $K\in\bSTi$ and the map $$\bSTi\rightarrow\ex(\cMEH), K\mapsto \law^{\text{ih}}(K)$$ is surjective and continuous.
		\end{enumerate} 
	\end{corollary}

	\begin{figure}[h]
		\centering
		\includegraphics[width=0.95\textwidth]{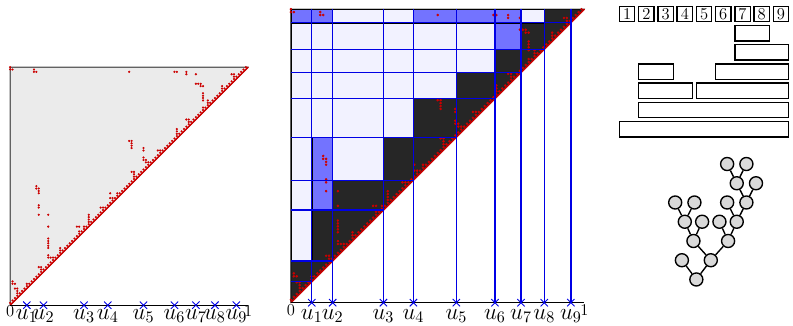}
		\caption{On the left a realization of $\color{red}n^{-1}I_n\color{black}$ for $I_n\sim \unif(\bBT(n))$ with $n=100$. On the right the binary tree $\phi^n_9(\color{red}n^{-1}I_n\color{black},\color{blue}u_1\color{black},\color{blue}\dots\color{black},\color{blue}u_9\color{black})$ for some $(\color{blue}u_1\color{black},\color{blue}\dots\color{black},\color{blue}u_9\color{black})\in[0,1]^9_<$, once pictured as a set of intervals and once in in the usual way as a tree.}
		\label{fig:phiinftyB}
	\end{figure}

	If one identifies every point $(x,y)\in K\in\bSTi$ with the open interval $(x,y)\subseteq (0,1)$ and the diagonal points with singletons, then one can regard every Schr{\"o}der tree on $(0,1)$ as a interval hierarchy on $(0,1)$ and hence $\bSTi\subseteq \bIHH$. The space $\bSTi$ is much 'smaller' and more structured then the large space $\bIHH$, since $(\bSTi,\dha)$ is a compact metric space. Although we reduced the cardinality of the space used to describe all ergodic exchangeable laws, our representation is far from unique as well; many different elements in $\bISi$ describe the same ergodic exchangeable hierarchy. Note that we obtained this result without using real trees. 
	
	\subsection{Binary trees}
	
	A binary tree on $[n]$ is a rooted ordered trees with exactly $n$ leafs in which every internal node has exactly two descendants (as a consequence, there are $n-1$ internal nodes). Thus we can introduce binary trees as subsets of Schr{\"o}der trees: A tree is binary iff for all choices of three disjoint subtrees there exists a fourth subtree that includes exactly two of the former and is disjoint to the third. This can be checked for leaves, which are the subtrees of size one. We present the following equivalent definition for binary trees, first in the finite case and then in the limit:
	\begin{definition}
		A Schr{\"o}der tree $\hT\in\bST(n)$ is called a \emph{binary tree}, if for every $1\leq j_1<j_2<j_3\leq n$ there is some $[a,b]\in\hT$ with either $a\leq j_1<j_2\leq b<j_3$ or $j_1<a\leq j_2<j_3\leq b$. Let $\bBT(n)\subseteq\bST(n)$ be the set of all binary trees on $[n]$. 
	\end{definition}

	\begin{definition}
		Let $U_1,U_2,U_3$ be iid uniform RVs. An element $K\in\bSTi$ is called \emph{binary tree on $(0,1)$} if $\phi^{\infty}_3(K,U_{1:3},U_{2:3},U_{3:3})$ is almost surely a binary tree on $[3]$.
	\end{definition}

	\begin{lemma}\label{lemma:bbtree}
		Binary trees form a substructure of Schr{\"o}der trees that satisfy the following consistency properties:
		\begin{enumerate}
			\item[i)] $\bBTi$ is a closed subset of $\bSTi$.
			\item[ii)] If $K\in\bBTi, (U_i)_{i\in\bN}$ is an $U$-process and $k\in\bN$ then $\phi^{\infty}_k(K,U_{1:k},\dots,U_{k:k})\in\bBT(k)$ almost surely. 
			\item[iii)] If $(\hT_n)_{n\in\bN}$ is a sequence with $\hT_n\in\bBT(n)$ and such that $n^{-1}\hT_n$ converges towards some $K$, then $K\in\bBTi$. 
		\end{enumerate}
	\end{lemma}
	\begin{proof}
		i):~Let $K_n$ be a sequence in $\bBTi$ converging towards some $K\in\bSTi$. By Lemma \ref{lemma:continuity} the sequence $T_n=\phi^{\infty}_3(K_n,U_{1:3},U_{2:3},U_{3:3})$ converges almost surely towards $T=\phi^{\infty}_3(K,U_{1:3},U_{2:3},U_{3:3})$. Since all $T_n$ are almost surely binary by definition, so is $T$ and hence $K\in\bBTi$.
		
		ii):~This follows from the fact that a finite tree $\hT\in\bST(k)$ is binary iff $\phi^k_3(\hT,\vj)$ is binary for every $\vj\in[3:k]$.
		
		iii):~Given any $\hT\in\bBT(n)$ define the set of points $A=\{(a-1)/n:[a,b]\in \hT\}\cup\{b/n:[a,b]\in\hT\}$. Now if $(u_1,u_2,u_3)\in[0,1]^3_<$ is such that $u_i\notin A$ for $i=1,2,3$ and $|u_1-u_3|>2/n$, then $\phi^{\infty}_3(\hT,u_1,u_2,u_3)\in\bBT(3)$. Now let $\hT_n\in\bBT(n)$ and $U_1,U_2,U_3$ be iid uniform on $[0,1]$. Let $A_n$ be the $A$-set corresponding to $\hT_n$ and $B=\cup_nA_n$. Since $B$ is countable, $U_{i:3}\notin B$ almost surely for all $i=1,2,3$. Almost surely there is an $N$ such that $|U_{1:3}-U_{3:3}|>2/n$ for all $n\geq N$. Consequently, $\phi^{\infty}_3(n^{-1}\hT_n,U_{1:3},U_{2:3},U_{3:3})$ is almost surely binary for all but finitely many $n$. Hence with Lemma \ref{lemma:continuity} the limit $\phi^{\infty}_3(K,U_{1:3},U_{2:3},U_{3:3})$ is almost surely binary and so $K\in\bBTi$. 
	\end{proof}	
	
	Again one obtains as a special case of Theorem \ref{thm:mainEIP} an almost sure characterization of erased-binary tree processes:
	
	\begin{corollary}\label{thm:mainEBT}
		For every $K\in\bBTi$ one has $\law(K)\in\ex(\cMB)$ and the map $\bBTi\rightarrow\ex(\cMB),K\mapsto \law(K)$ is a homeomorphism. One has the following concrete representation: Let $(T,\eta)=(T_n,\eta_n)_{n\in\bN}$ be an erased-binary tree process. Then $n^{-1}T_n$ converges almost surely as $n\rightarrow\infty$ towards some $\bBTi$-valued random variable $T_{\infty}$. Let $U=(U_i)_{i\in\bN}$ be the $U$-process corresponding to $\eta$. Then $T_{\infty}$ and $U$ are independent and one has the equality of processes
		\begin{equation*}
		(T_n,\eta_n)_{n\in\bN}~=~\Big(\phi^{\infty}_n(T_{\infty},U_{1:n},\dots,U_{n:n}),\eta_n\Big)_{n\in\bN}~~\text{almost surely.}
		\end{equation*}
		In particular, for every erased-binary tree process $(T,\eta)$ the conditional law of $(T,\eta)$ given the terminal $\sigma$-field $\cF_{\infty}$ is $\law(T_{\infty})$ almost surely and $T_{\infty}$ generates $\cF_{\infty}$ almost surely.
	\end{corollary} 

	\begin{remark}
		One could introduced exchangeable \emph{binary} hierarchies on $\bN$: some exchangeable hierarchy $H=(H_n)_{n\in\bN}$ on $\bN$ is called binary if every $H_n$ is binary as a tree. One could have obtained as a consequence of Corollary \ref{thm:mainEBT} the exact analogue of Corollary \ref{thm:mainEIHier} with $\bBTi$ instead of $\bSTi$.  
	\end{remark}

	\subsection{Martin boundaries and limits of ordered discrete structures}
	
	We will give a very short definition of Martin boundary that is adapted best to our already used choice of symbols and refer the reader to \cite{choi2017doob, egw2, gerstenberg, vershik2015equipped} for more details. We introduce this concept for interval systems first and then relate this to Martin boundaries associated with Schr{\"o}der trees and finally with binary trees.  
	
	For any $1\leq k\leq n$ let $\vec{J}$ be a random vector uniformly distributed on $[k:n]$. For any $\hI_k\in\bIS(k)$ and $\hI_n\in\bIS(n)$ define 
	\begin{equation}\label{eq:backwardsprob}
	\gamma(\hI_k,\hI_n):=\bP(\phi^n_k(\hI_n,\vec{J})=\hI_k)=\frac{1}{\binom{n}{k}}\#\Big\{\vj\in[k:n]:\phi^n_k(\hI_n,\vj)=\hI_k\Big\}
	\end{equation}
	and for $k>n$ set $\gamma(\hI_k,\hI_n):=0$. The value $\gamma(\hI_k,\hI_n)$ is obtained by counting how often the smaller interval system $\hI_k$ is \emph{embedded} into the larger interval system $\hI_n$ and divides this amount by the maximal possible number of such embeddings. One can think of $\gamma(\hI_k,\hI_n)$ to be the \emph{density of the small $\hI_k$ in the large $\hI_n$}. This interpretation is in line with a emerging field in the area of limits of combinatorial objects, most famously discussed for graph limits. The connection to exchangeability and related areas is a commonly used tool that helps to understand the limiting behaviors of such density numbers as the size of the large object tends to infinity, see \cite{dija, austin, lovasz, hkmrs}. 
	
	The object $\gamma$ and erased-interval processes are linked as follows: For any erased-interval process $(I_n,\eta_n)_{n\in\bN}$ the first coordinate process $(I_n)_{n\in\bN}$ is a \emph{Markov chain with co-transition probabilities $\gamma$}, that is $\bP(I_k=\hI_k|I_n=\hI_n)=\gamma(\hI_k,\hI_n)$ for all $1\leq k\leq n, \hI_k\in\bIS(k)$ and $\hI_n\in\bIS(n)$ with $\bP(I_n=\hI_n)>0$. By Kolmogorov's extension theorem the opposite is true in the following sense: To any Markov chain $(\hat I_n)_{n\in\bN}$ with $\hat I_n\in\bIS(n)$ and co-transition probabilities given by $\gamma$ there is a unique (in law) erased-interval process $(I_n,\eta_n)_{n\in\bN}$ such that $(\hat I_n)_{n\in\bN}$ and $(I_n)_{n\in\bN}$ have the same distribution. 
	
	Given any sequence $(\hI_n)_{n\in\bN}$ of interval systems with $\hI_n\in\bIS(m_n)$ for some sequence $m_n\rightarrow\infty$ one says that this sequence is \emph{$\gamma$-convergent} iff $\gamma(\hI,\hI_n)$ converges as $n\rightarrow\infty$ for every $k,\hI\in\bIS(k)$. We think of the pointwise defined functions $\lim\nolimits_{n\rightarrow\infty}\gamma(\cdot,\hI_n):\bigcup\nolimits_{k\geq 1}\bIS(k)\rightarrow[0,1]$ to be the limit objects associated to $\gamma$-convergent sequences. The set of all functions $\bigcup\nolimits_{k\geq 1}\bIS(k)\rightarrow[0,1]$ obtainable in this way constitutes the \emph{Martin boundary associated to $\gamma$}. This Martin boundary can be described equivalently as a set of laws: For any $\gamma$-convergent sequence $(\hI_n)_{n\in\bN}$ there exists a unique (in law) erased-interval process $(I_n,\eta_n)_{n\in\bN}$ such that 
	\begin{equation}\label{eq:convergence}
	\bP(I_k=\hI_k)=\lim\nolimits_{n\rightarrow\infty}\gamma(\hI_k,\hI_n)~~\text{for all}~k\in\bN,\hI_k\in\bIS(k).
	\end{equation}
	This again is a consequence of Kolmogorov's existence theorem. We identify the Martin boundary associated to $\gamma$ with the set of all laws of EIPs that fulfill (\ref{eq:convergence}) for some $\gamma$-convergent sequence and we define this set as $\partial(\cMI)$. So in particular, $\partial(\cMI)\subseteq\cMI$. General theory yields that $\partial(\cMI)$ is always a closed subset of $\cMI$ and that every extreme point of $\cMI$ is a point in the Martin boundary as well, so $\ex(\cMI)\subseteq \partial(\cMI)\subseteq\cMI$. It is often the case that extreme points and Martin boundary coincide and this is also the case here: We will not present a proof of this fact but direct the reader to \cite{gerstenberg}, Lemma $4.2.$. There it was shown that the Martin boundary and the set of extreme points coincide in the context of (general) erased-word processes. The proof given there only uses exactly the same features that are present for interval systems. In fact, the proof presented in \cite{gerstenberg} was largely inspired by the proof presented in \cite{egw2} showing that Martin boundary and extreme points coincide in the context of R\'emy's tree growth chain. Given the equality of extreme points and Martin boundaries we directly obtain the following corollary to Theorem \ref{thm:mainEIP}:
	
	\begin{corollary}
		Let $(U_i)_{i\in\bN}$ be a $U$-process. For any $\gamma$-convergent sequence $(\hI_n)_{n\in\bN}$ of interval systems there exists a unique $K\in\bISi$ such that 
		$$\lim\nolimits_{n\rightarrow\infty}\gamma(\hI,\hI_n)=\bP(\phi^{\infty}_k(K,U_{1:k},\dots,U_{k:k})=\hI)$$
		holds for every $k\in\bN, \hI\in\bIS(k)$. This map $\bISi\rightarrow \partial(\cMI)$ yields a homeomorphic description of the Martin boundary of interval systems with respect to $\gamma$.
	\end{corollary}

	Since Schr{\"o}der trees and binary trees are both stable under removing elements according to $\phi^n_k$ one directly obtains the following
	
	\begin{corollary}
		Let $(U_i)_{i\in\bN}$ be a $U$-process. For any $\gamma$-convergent sequence $(\hT_n)_{n\in\bN}$ of Schr{\"o}der trees there exists a unique $K\in\bSTi$ such that 
		$$\lim\nolimits_{n\rightarrow\infty}\gamma(\hT,\hT_n)=\bP(\phi^{\infty}_k(K,U_{1:k},\dots,U_{k:k})=\hT)$$
		holds for every $k\in\bN, \hT\in\bST(k)$. This map $\bSTi\rightarrow \partial(\cMS)$ yields a homeomorphic description of the Martin boundary of Schr{\"o}der trees with respect to $\gamma$.
	\end{corollary}
	
	\begin{corollary}\label{thm:mbb}
		Let $(U_i)_{i\in\bN}$ be a $U$-process. For any $\gamma$-convergent sequence $(\hT_n)_{n\in\bN}$ of binary trees there exists a unique $K\in\bBTi$ such that 
		$$\lim\nolimits_{n\rightarrow\infty}\gamma(\hT,\hT_n)=\bP(\phi^{\infty}_k(K,U_{1:k},\dots,U_{k:k})=\hT)$$
		holds for every $k\in\bN, \hT\in\bBT(k)$. This map $\bBTi\rightarrow \partial(\cMB)$ yields a homeomorphic description of the Martin boundary of binary trees with respect to $\gamma$.
	\end{corollary}

	\begin{exmp}
		In \cite{egw2} two examples of $\gamma$-convergent sequences of binary trees $(\hT_n)_{n\in\bN}$ were considered: 
		\begin{enumerate}
			\item Spine trees: $\hT_n\in\bBT(n)$ is binary of hight $n-1$ that grows from the root left-right-left-right-\dots.
			\item Complete trees: $\hT_n\in\bBT(2^n)$ is the complete binary tree of hight $n$. 
		\end{enumerate}
		Both sequences are $\gamma$-convergent. The limit of spine trees is given by $$K=\{(x,1-x):0\leq x\leq 0.5\}\cup\diago{}$$ and the limit of complete trees is given by $$K=\bigcup_{n\geq 0}\Big\{\left(0,\frac{1}{2^n}\right),\left(\frac{1}{2^n},\frac{2}{2^n}\right),\left(\frac{2}{2^n},\frac{3}{2^n}\right),\dots,\left(\frac{2^n-1}{2^n},1\right)\Big\}\cup\diago{}.$$
		\end{exmp}

	\begin{remark}
		A general theory that can be applied to prove the equality of Martin boundaries and extreme points in the present situations is presented in the author's Ph.D. thesis \cite{gerstenbergdiss}. 
	\end{remark}

	\begin{remark}
		One should be able to prove that a sequence $(\hI_n)_{n\in\bN}$ with $\hI_n\in\bIS(m_n)$ and $m_n\rightarrow\infty$ is $\gamma$-convergent iff $m_n^{-1}\hI_n$ converges in $(\bISi,\dha)$. 
	\end{remark}

	Next we present R\'emy's tree growth chain and translate the notion of Martin boundary used in \cite{egw2} to our situation. R\'emy's tree growth chain is a Markov chain $(T_n)_{n\in\bN}$ with $T_n\in\bBT(n)$ for every $n$, that can be obtained as follows: $T_1\in\bBT(1)$ is the unique binary tree consisting only of a root vertex. The transitions are as follows: Given some binary tree $\hT\in\bBT(n)$ one chooses one of the $2n-1$ nodes in $\hT$ uniform at random. Let $v$ be that node. Then one cuts the subtree with root $v$ off and puts it aside. At the position of $v$ one places the unique binary tree with two leaves (the 'cherry'). The put-aside subtree then is placed at one of the two leaves of the cherry, chosen with equal probability. The resulting tree is binary by construction and has $n+1$ leaves. The resulting process $T=(T_n)_{n\in\bN}$ is R\'emy's tree growth chain. In \cite{egw2} it was shown that $T$ has co-transition probabilities $\gamma$ and each $T_n$ is \emph{uniform} on $\bBT(n)$. The Martin boundary associated to R\'emy's Tree Growth chain described in \cite{egw2} is equivalent to the above introduced Martin boundary of binary trees associated to $\gamma$. To obtain a description of that Martin boundary the authors first used the Kolmogorov existence theorem to construct \emph{labeled infinite R\'emy bridges}. Such an object is basically a process $(T_n,S_n)_{n\in\bN}$ according to an erased-binary tree process $(T_n,\eta_n)_{n\in\bN}$ in which $S$ is the permutation process according to $\eta$. The variables $\eta$ appeared in their work as well, but were called '$L_i$' instead of '$\eta_i$', see \cite[Lemma 5.3.]{egw2}. Arriving at labeled infinite R\'emy bridges (equivalent to erased-binary tree processes) they constructed what they called \emph{exchangeable didendritic systems}, certain infinite combinatorial exchangeable objects, see \cite[Definition 5.8.]{egw2}. They then reduced the problem of describing the Martin boundary to the task of describing ergodic didendritic systems by showing that Martin boundary and extreme points coincide. For every ergodic didendritic system they introduce certain indicator arrays that inherit exchangeability and then employ the classical de Finetti theorem to obtain almost sure convergence. The resulting almost sure limits were then used to construct a \emph{binary real tree with a probability measure concentrated on the leafs} ('Rt' for short). The construction of that Rt basically depends on $(S_n(T_n))_{n\in\bN}$ alone, which  according to Theorem \ref{thm:mainMAP} is an exchangeable (binary) hierarchy on $\bN$. Consequently, the construction of the Rt takes place in the same situation as it has been done in \cite{fohapi}, which explains the resemblances of the two papers. However, in contrast to \cite{fohapi}, the random finite objects under consideration are now \emph{ordered} trees, the transition from $\cL((T_n,S_n)_{n\in\bN})$ to $\cL((S_n(T_n))_{n\in\bN})$ is not injective. As a consequence, the constructed Rt is not sufficient to distinguish all points in the Martin boundary associated to binary trees. The missing information was described in \cite[Section 8]{egw2} and used an higher order Aldous-Hoover-Kallenberg representation result. However, higher-order randomization is not really needed, see \cite[Lemma 8.1.]{egw2}. The final representation result they obtain is a surjective description of the Martin boundary of R\'emy's tree growth chain, see \cite[Theorem 8.2.]{egw2}. A similar approach is applied in \cite{choi2017doob, evans2016radix} to describe Martin boundaries in different situations. 
	
	\subsection{Compositions}
	
	We present a further examples of a nicely embedded substructure: compositions. In \cite{gnedin} \emph{exchangeable compositions of $\bN$} have been analyzed. A composition of $\bN$ is a tuple $C=(P,l)$, where $P=\{e_1,e_2,\dots\}$ is a partition of $\bN$ and $l$ is a linear order on $P$. Given any finite bijection $\pi\in\bS_{\infty}$ one can relabel the composition: $\pi(C):=(\pi(P),\pi(l))$ where $\pi(P)=\{\pi(e_1),\pi(e_2),\dots\}$ and $\pi(e_i)\pi(l)\pi(e_j):\Leftrightarrow e_ile_j$. Now $\pi(C)$ is a composition as well. One can topologize the space of compositions via finite restrictions: Given any $n\in\bN$ one introduces $C_{|n}$ where the partition is restricted to $[n]$ and the linear order is restricted to the images of the remaining partition blocks. An \emph{exchangeable composition of $\bN$} is a random composition $\Pi$ of $\bN$ such that $\pi(\Pi)\sim\Pi$ for every $\pi\in\bS_{\infty}$. Let $\cMC$ be the simplex of laws of exchangeable compositions. Gnedin \cite{gnedin} obtains a homeomorphic description of $\ex(\cMC)$ in terms of open subsets of the unit interval. The set of open subsets was topologized in an explicit way that turns the description into a homeomorphic one.

	This is linked to erased-interval processes as follows: An \emph{interval partition} of $[n]$ is an element $\hP\in\bIS(n)$ such that all non-singleton intervals $[a,b],[c,d]\in\hP$ are disjoint. Let $\bIP(n)$ be the set of all interval partitions of $[n]$. As above, the family $\bIP(n), n\in\bN$ is stable under sampling via $\phi^n_k$. In \cite[p. 1439]{gnedin} it is explained that describing the simplex $\cMC$ is equivalent (affinely homeomorphic) to describing the simplex of all laws of Markov chains of growing interval-partitions $(P_n)_{n\in\bN}$ with co-transition probabilities $\gamma$. Now this simplex of laws is equivalent to the simplex of laws of erased-interval partition processes. The latter can be described by the subspace $\bIP(\infty)\subseteq \bISi$ consisting of all $K$ such that for all $(x,y),(x',y')\in K$ with $x<y$ and $x'<y'$ it holds that either $y\leq x'$ or $y'\leq x$. The homeomorphism is analogous to the Corollaries \ref{thm:mainEST} and \ref{thm:mainEBT}. Now every open subset $U$ of the unit interval is a countable union of disjoint open intervals, $U=\cup_{i}(x_i,y_i)$. If one reads every $(x_i,y_i)$ as a point in $\recto{}$ and one then passes from $U$ to $\{(x_i,y_i):i\}\cup\diago{}$ one obtains a homeomorphism between the space describing $\ex(\cMC)$ in \cite{gnedin} and the space $(\bIP(\infty),\dha)$. 
	
	\subsection{Outlook}
	
	Let $T=(T_n)_{n\in\bN}$ be R\'emy's tree growth chain. It is part of our main theorem that $n^{-1}T_n$ converges almost surely in the space $(\bSTi,\dha)$ towards some $\bBTi$-valued random variable $T_{\infty}$. 
	
	\begin{definition}
		$T_{\infty}$ is called the \emph{Brownian binary tree on $(0,1)$}.
	\end{definition}
	
	It is a well known fact that the normalized \emph{exploration paths} associated with $(T_n)_{n\in\bN}$ converge almost surely towards a (the same for all) Brownian excursion $E$, see \cite{marchal} and \cite{mm}. In particular, the law of R\'emy's tree growth chain is not ergodic, since Brownian excursion is not degenerate. It is known that the terminal $\sigma$-field generated by $T$ is almost surely equal to the $\sigma$-field generated by $E$ and the same is true for the $\sigma$-field generated by $T_{\infty}$. Hence $\sigma(T_{\infty})=\sigma(E)$ almost surely: One can express $T_{\infty}$ as a measurable function of $E$ and vice versa, almost surely. In future work we aim to describe this in detail. We plan to compare the map $\phi^{\infty}_k:\bSTi\times[0,1]^k_<\rightarrow\bST(k)$ with maps used to build trees from \emph{excursions} (see \cite[Chapter III, Section 3.]{legall}). Excursions are commonly used to describe real trees, we will present a way to describe certain Schr{\"o}der trees on $(0,1)$ by excursion and compare these to the real tree obtained from the same excursion.
	
	We plan to investigate some possible universality properties of (the law of) $T_{\infty}$: There are well known connections of large uniform pattern avoiding permutations and Brownian excursion (see \cite{horisl}), which may be explained with the help of the Brownian tree on $(0,1)$. We finish our considerations concerning trees with the following conjecture, which is due to a comparison of the left hand side of Fig. \ref{fig:phiinftyB} and \cite[Figure 3]{horisl}: 
	
	\begin{conjecture}
		Let $n\in\bN$ and $S_n$ be a uniform $231$-avoiding permutation of $[n]$. Consider the normalized graph of $S_n$, defined as $\text{gr}(S_n):=\{(i/n,S_n(i)/n):i\in[n]\}$. We consider $\text{gr}(S_n)$ as a random compact set, thus taking values in the space $(\cK([0,1]^2),\dha)$. As $n\rightarrow\infty$ the sequence $\text{gr}(S_n)$ converges in law towards the Brownian binary tree on $(0,1)$, where convergence in law is with respect to the Hausdorff topology $\dha$ on $\cK([0,1]^2)$.  
	\end{conjecture}
	
	We were able to deduce certain results for subclasses of interval systems, since they are included in a 'nice way' that is somehow consistent with our most important operations, see Lemmas \ref{lemma:bst} and \ref{lemma:bbtree}. In future research one could look for more such 'nicely embedded' substructures of interval systems.

	\begin{acknowledgements}
		The author would like to thank his Ph.D. supervisor\\ Rudolf Gr{\"u}bel for many interesting discussions during the last years and for a lot of very helpful comments concerning this paper.
	\end{acknowledgements}

\end{document}